\documentclass[a4paper,final]{article}
\usepackage{amsmath,bm}
\usepackage{amssymb}
\usepackage{amsthm}
\usepackage{mathtools}
\usepackage{mathrsfs}
\usepackage[utf8]{inputenc}
\usepackage{cases}
\usepackage[dvipdfmx]{graphicx,color}
\usepackage{enumitem,kantlipsum}
\usepackage[numbers]{natbib}
\RequirePackage[colorlinks,citecolor=blue,urlcolor=blue,backref=page]{hyperref}
       \usepackage{datetime2}

\makeatletter
\newcommand{\subjclass}[2][1991]{%
  \let\@oldtitle\@title%
  \gdef\@title{\@oldtitle\footnotetext{#1 \emph{Mathematics subject classification.} #2}}%
}
\newcommand{\keywords}[1]{%
  \let\@@oldtitle\@title%
  \gdef\@title{\@@oldtitle\footnotetext{\emph{Key words and phrases.} #1.}}%
}
\makeatother

\theoremstyle{plain}
\newtheorem{thm}{Theorem}[section]
\newtheorem{lem}[thm]{Lemma}
\newtheorem{cor}[thm]{Corollary}
\newtheorem{prop}[thm]{Proposition}

\theoremstyle{definition}
\newtheorem{df}[thm]{Definition}
\newtheorem{eg}[thm]{Example}

\theoremstyle{remark}
\newtheorem{rem}[thm]{Remark}

\newcommand{\remove}[1]{}
\newcommand{\qad}{\phantom{={}}}
\newcommand{\relmiddle}[1]{\mathrel{}\middle#1\mathrel{}}
\numberwithin{equation}{section}
\allowdisplaybreaks[3]

\newcommand{\E}{\mathbb{E}}

\newcommand{\K}{\mathbb{K}}
\newcommand{\N}{\mathbb{N}}
\renewcommand{\P}{\mathbb{P}}

\newcommand{\R}{\mathbb{R}}

\newcommand{\X}{\mathbb{X}}
\newcommand{\Z}{\mathbb{Z}}
\newcommand{\1}{\mbox{\rm1}\hspace{-0.25em}\mbox{\rm l}}

\newcommand{\cK}{\mathcal{K}}
\newcommand{\cL}{\mathcal{L}}

\newcommand{\cN}{\mathcal{N}}

\newcommand{\cW}{\mathcal{W}}

\newcommand{\bfb}{\mathbf{b}}
\newcommand{\bfB}{\mathbf{B}}

\renewcommand{\a}{\alpha}
\renewcommand{\b}{\beta}
\newcommand{\gm}{\gamma}

\newcommand{\dl}{\delta}
\newcommand{\Dl}{\Delta}
\newcommand{\eps}{\varepsilon}
\newcommand{\lm}{\lambda}

\newcommand{\sg}{\sigma}

\newcommand{\Om}{\Omega}

\newcommand{\ER}{Erd\H{o}s--R\'{e}nyi }

\DeclareMathOperator{\ber}{Bernoulli}
\DeclareMathOperator{\bin}{Bin}
\DeclareMathOperator{\cov}{Cov}

\DeclareMathOperator{\var}{Var}
\DeclareMathOperator{\supp}{supp}

\usepackage{comment}

\newcommand{\dy}[1]{\textcolor{magenta}{#1}}
\newcommand{\tNN}{\text{NN}}

\title{
Normal approximation for statistics of randomly weighted complexes
}

\author{
Shu \textsc{Kanazawa}\thanks{Kyoto University Institute for Advanced Study, Kyoto University, Japan.
\texttt{kanazawa.shu.2w@kyoto-u.ac.jp}}
\and
Khanh Duy \textsc{Trinh}\thanks{Global Center for Science and Engineering, Waseda University, Japan. \texttt{trinh@aoni.waseda.jp}}
\and
D. \textsc{Yogeshwaran}\thanks{Theoretical Statistics and Mathematics Unit, Indian Statistical Institute, Bangalore.
\texttt{d.yogesh@isibang.ac.in}}
}



\date{}

\begin{document}

\maketitle

\begin{abstract}
We prove normal approximation bounds for statistics of randomly weighted (simplicial) complexes. In particular, we consider the complete $d$-dimensional complex on $n$ vertices with $d$-simplices equipped with i.i.d.\ weights.
Our normal approximation bounds are quantified in terms of stabilization of difference operators, i.e., the effect on the statistic under addition/deletion of simplices. Our proof is based on Chatterjee's normal approximation bound and is a higher-dimensional analogue of the work of Cao on sparse \ER random graphs but our bounds are more in the spirit of `quantitative two-scale stabilization' bounds by Lachi\`eze-Rey, Peccati, and Yang. As applications, we prove a CLT for nearest face-weights in randomly weighted $d$-complexes and give a normal approximation bound for local statistics of random $d$-complexes.      
\end{abstract}

\section{Introduction}

This article considers normal approximations and convergences for combinatorial optimization problems on randomly weighted complexes. A first-order theory or laws of large numbers for combinatorial optimization problems on sparse random graphs is well-developed via local weak convergence \cite{AS04,vdH16,vdH17} or interpolation method \cite{Sa16}. A similar theory for random simplicial complexes or hypergraphs has been developed in \cite{Ka22,LiP16,LiP17} and \cite{DA18}, respectively.
However,  a second-order theory for such models is in a nascent stage with some efforts for central limit theorem (CLT) for statistics of \ER random graphs and configuration models \cite{AY20,BR19,Ja20,ERTZ21}. Cao \cite{Ca21} investigated normal approximations for combinatorial optimization problems on sparse \ER random graphs. In this article, we derive an abstract normal approximation result for random complexes borrowing Cao's approach but the actual bounds are analogous to the so-called `quantitative two-scale stabilization' bounds in \cite{LPY20}. More specifically, Cao's bounds necessitated quantifying the convergence of \ER random graphs to Galton--Watson tree but our bounds involve comparison of difference operators in two different scales on the random complex alone and this is more useful in our computations. This allows us to circumvent the task of quantifying the convergence of random complexes to acyclic complexes (higher-dimensional analogue of trees). We use our abstract normal approximation result to prove the normal convergence for some specific statistics. We now present an illustrative example and then present a preview of our general theorem along with discussions on related literature and proof ideas. 

Consider the following statistic. Let $\K_n$ be the $d$-skeleton
of the complete complex $\cK_n$ on $[n] \coloneqq \{1,2,\ldots,n\}$ with i.i.d.\ weights $\{w_{n,\tau}\colon\tau\in F_d(\cK_n)\}$ sampled from $\text{Exp}(1/n)$, the exponential distribution with mean $n$. Here, $F_i(\cK_n)$ denotes the set of all $i$-simplices in $\cK_n$, i.e., the $(i+1)$-tuples in $[n]$. (For definitions, refer to Section \ref{sec:model-result}.) For any $\sg\in F_{d-1}(\cK_n)$, we set {\em the nearest face-weight} of $\sg$ as
\begin{equation}
\label{e:nearfaceweight}
\text{NN}_n(\sg)\coloneqq\min_{\sg\subset\tau\in F_d(\cK_n)}w_{n,\tau}
\end{equation}
and so the total nearest face-weight $\text{NN}_n(\K_n)$ of $\K_n$ is defined by
\begin{equation}
\label{e:totnearfaceweight}
\text{NN}_n(\K_n)\coloneqq\sum_{\sg\in F_{d-1}(\cK_n)} \text{NN}_n(\sg).
\end{equation}
In the case $d = 1$, $\tNN$ is the sum of nearest neighbour weights on the complete graph with i.i.d.\ edge weights. It is a statistic of interest in probabilistic combinatorial optimization \cite{Steele97,Yukich06} and is also closely related to the minimal spanning tree. In fact, for all $d \geq 1$,  $\tNN$ is related to minimal spanning acycles, a topological generalization of minimal spanning trees; see for example, \cite{HS17,PSY20}. The law of large numbers for nearest neighbour weights (i.e., $\tNN$ in $d =1$) and the minimal spanning tree can be obtained using the results of \cite{AS04}. A CLT for minimal spanning tree was proven in \cite{Ja95} and it should be possible to do the same for nearest neighbour weights.  One can also deduce a CLT for nearest neighbour weights from the recent results of \cite{Ca21}. One of our applications is a CLT for nearest face-weights, i.e., $\tNN_n(\K_n)$ for all $d \geq 1$. For two non-negative real-valued sequences $a_n$ and $b_n$, we say that $a_n \sim b_n$ if $a_n/b_n \to 1$ as $n \to \infty$.
\begin{thm}\label{thm:CLT_NN}
It holds that
\[
\var(\text{NN}_n(\K_n)) \sim \Big( 1+\frac{d}{2} \Big)\binom{n}{d} \quad \mbox{and} \quad \frac{\text{NN}_n(\K_n)-\E[\text{NN}_n(\K_n)]}{\sqrt{\var(\text{NN}_n(\K_n))}}\xrightarrow[n\to\infty]{d}\cN(0,1),
\]
where $\cN(0,1)$ is a standard normal variable, and $\xrightarrow[n\to\infty]{d}$ indicates the convergence in distribution as $n\to\infty$.
\end{thm}
A related statistic is the number of isolated simplices at level $\alpha$, i.e., $\sum_{\sg\in F_{d-1}(\cK_n)}\1[\text{NN}_n(\sg) > \alpha],  \alpha \in (0,\infty)$.  This is nothing but number of isolated $(d-1)$-simplices in the random $d$-complex (also known as Linial--Meshulam--Wallach complex) on $n$ vertices with every $d$-simplex chosen with probability $1 - e^{-\alpha/n}$. A normal approximation for that statistic was recently derived in \cite[Theorem 1.4]{ERTZ21} using discrete Malliavin--Stein method for Rademacher random variables. It can also be derived using our main theorems (see Example \ref{eg:isolface}) and also one can derive a normal approximation for weighted local statistics of the random $d$-complex (see Section~\ref{sec:localstat}). It is not immediate if the CLT for weighted random $d$-complex can be derived from the results of \cite{ERTZ21} or extensions of the same. CLTs and normal approximations for some local statistics of random clique complexes can be found in \cite{KM13,TNR22}, which are proven using Stein's method for $U$-statistics and sums of dissociated random variables. It would be worthwhile to extend our general theorems to cover weighted random clique complexes as well. When discussing our general theorems and proof approach below,  we shall give pointers to CLTs for Euclidean (lattice and continuum) random complexes.

We will now discuss our general theorem, proof approach, and some related literature. We consider the following randomly weighted complex.  Let $\X_n$ be a random complex on $n$ vertices with the complete $(d-1)$-skeleton and every $d$-simplex is chosen with a probability $p_n$ as well as equipped with i.i.d.\ non-negative weights. The aim of this article is to investigate a CLT for $f(\X_n)$, where $f$ is a functional on weighted $d$-complexes. Towards this, we prove a normal approximation for 
$$ \frac{f(\X_n)- \E[f(\X_n)]}{\sqrt{\var(f(\X_n))}}.$$  
Our CLT is driven by suitably formalizing the notion of stabilization for randomly weighted complexes. The normal approximation theorem relies on quantifying stabilization via difference operators of $f$, i.e., quantifying the effect of adding/deleting some $d$-simplices on $f(\X_n)$. 

Stabilization via difference operators was identified as a key to show CLTs in Euclidean settings in \cite{Pe01,PY01,Tr19} and this has been highly useful in proving CLTs for Euclidean random complexes; see \cite{YSA17,HST18,HT18,KP19,SY22}. The analogue of such a result in greater generality is still missing in \ER type random graph models and related random complex models. \cite{Ch08} introduced a new method of normal approximation which involved quantifying the stabilization of difference operators in two-scales and has now inspired CLTs for statistics of Euclidean random structures; see \cite{Ch14,LPS16,CS17,LPY20}. This has been adapted to include statistics of weighted \ER random graphs in \cite{Ca21}, which is the starting point of our approach. While our approach is similar to that of \cite{Ca21},  our bounds are of a different nature. They are given intrinsically in terms of the random complex itself in the spirit of `quantiative two-scale stabilization' bounds in \cite{LPY20} whereas the bounds in \cite[Theorem 2.4]{Ca21} are quantified by comparing the statistic on random graphs with that on an appropriate randomly weighted Galton--Watson tree, the local weak limit of \ER random graphs. Apart from simplifying certain computations in our applications, the other reason to do so is that the convergence rate of random complexes to random acycles is still not available. Though it may be possible to derive rates by refining the methods in \cite{Ka22,LiP16,LiP17}, our normal approximation result eschews the need for such rates. The question of proving even a CLT for more global functionals such as Betti numbers, persistent Betti numbers and the total weight of minimal spanning acycles remains open.

The rest of the article is organized as follows: In the next section (Section~\ref{sec:model-result}), we introduce our weighted random complex model and state our main normal approximation results and some consequences. In Section~\ref{sec:applns},  we give the two applications of our general results to nearest face-weights and local statistics. Proofs of our main results are deferred to Section~\ref{sec:proofs}. 

\section{Model and main results}\label{sec:model-result}
In this section, we introduce our weighted simplicial complex model (Section~\ref{ssec:wgt_cpx}) and state the general main results and some corollaries (Section~\ref{ssec:mainres}).

\subsection{Preliminaries on weighted complexes}
\label{ssec:wgt_cpx}
Let $V$ be a nonempty set.
A collection $X$ of finite subsets of $V$ is called a \textit{simplicial complex} on $V$ if $X$ satisfies (i)~$\{v\}\in X$ for all $v\in V$; and (ii)~$\tau\in X$ whenever $\sg \in X$ and $\tau\subset\sg$.
Note that all simplicial complexes must include the empty set.
Every element $\sg\in X$ is called a \textit{simplex in $X$}, and the \textit{dimension} of $\sg$ is defined by $\dim\sg\coloneqq\#\sg-1$.
A simplex $\sg\in X$ with $\dim\sg=k$ is referred to as a \textit{$k$-simplex in $X$}.
For $k\ge-1$, let $F_k(X)$ denote the set of all $k$-simplices in $X$, and set $f_k(X)\coloneqq\# F_k(X)$.
The \textit{degree} $\deg_X(\tau)$ of a $k$-simplex $\tau$ in $X$ is defined as the number of $(k+1)$-simplices in $X$ containing $\tau$.
A simplicial complex contained in $X$ is called a \textit{subcomplex} of $X$.
For $d\ge0$, the subcomplex $X^{(d)}\coloneqq\bigsqcup_{k=-1}^d F_k(X)$ is called the \textit{$d$-skeleton} of $X$.

Let $\cK_n$ be the \textit{complete simplicial complex} on $[n]\coloneqq\{1,2,\ldots,n\}$, i.e., $\cK_n\coloneqq 2^{[n]}$.
Throughout this article, we fix $d\in\N$, and call a simplicial complex $X$ such that $\cK_n^{(d-1)}\subset X\subset\cK_n^{(d)}$ a \textit{$d$-complex in $\cK_n$} for simplicity.
A \textit{weighted $d$-complex in $\cK_n$} is a pair of a $d$-complex $X$ in $\cK_n$ and a family $\{w_\tau\colon\tau\in F_d(X)\}$ of nonnegative values.
One convenient way to give a weighted $d$-complex in $\cK_n$ is the following: given $B=\{b_\tau\colon\tau\in F_d(\cK_n)\}$ and $W=\{w_\tau\colon\tau\in F_d(\cK_n)\}$ such that $b_\tau\in\{0,1\}$ and $w_\tau\in[0,\infty)$ for every $\tau\in F_d(\cK_n)$, we construct a weighted $d$-complex $\X_{(B,W)}$ in $\cK_n$ by starting with $\cK_n^{(d-1)}$ and adding each $d$-simplex $\tau\in F_d(\cK_n)$ with its weight $w_\tau$ if and only if $b_\tau=1$.
We call $\X_{(B,W)}$ a \textit{weighted $d$-complex generated by $(B,W)$}.

We shall consider a real-valued function $f$ on weighted $d$-complexes in $\cK_n$.
We will need to assume that $f$ is invariant under weighted isomorphism and Lipschitz in the following sense.
\begin{itemize}
\item \textit{Invariance under weighted isomorphism}.
Two weighted $d$-complexes $(X,\{w_\tau\colon\tau\in F_d(X)\})$ and $(X',\{w'_\tau\colon\tau\in F_d(X')\})$ in $\cK_n$ are said to be \textit{isomorphic} if there exists a permutation $\pi$ on $[n]$ such that for any $\tau\in F_d(\cK_n)$, $\tau\in X$ iff $\pi(\tau)\in X'$, and $w_\tau=w_{\pi(\tau)}$ holds in that case.
Here, $\pi(\tau)$ indicates the image of $\tau$ by $\pi$.
If $f(\X) = f(\X')$ whenever the weighted $d$-complexes $\X=(X,\{w_\tau\colon\tau\in F_d(X)\})$ and $\X'=(X',\{w'_\tau\colon\tau\in F_d(X')\})$ are isomorphic, then we say that $f$ is {\em invariant under weighted isomorphisms}.

\item \textit{Lipschitzness}.
$f$ is said to be {\em $H$-Lipschitz} if there exists a measurable function $H\colon[0,\infty)^2\to[0,\infty)$ such that for any $B=\{b_\tau\colon\tau\in F_d(\cK_n)\}$, $W=\{w_\tau\colon\tau\in F_d(\cK_n)\}$, $B'=\{b'_\tau\colon\tau\in F_d(\cK_n)\}$, and $W'=\{w'_\tau\colon\tau\in F_d(\cK_n)\}$,
\[
\bigl|f\bigl(\X_{(B,W)}\bigr)-f\bigl(\X_{(B',W')}\bigr)\bigr| \le\sum_{\substack{\tau\colon b_\tau\vee b'_\tau=1\\b_\tau\neq b'_\tau}}H(w_\tau,w'_\tau).
\]
Here, $a\vee b$ denotes $\max\{a,b\}$ for any $a,b\in\R$. 
\end{itemize}

In both $f$ and $H$, we have suppressed the dependence on $n$ for convenience.
The following notation will be repeatedly used in this article.
\begin{df}
\label{d:knbhd_complex}
Let $X$ be a $d$-complex in $\cK_n$.
We call a finite sequence $(\sg_0,\sg_1\ldots,\sg_k)$ of $(d-1)$-simplices a \textit{path in $X$} if $\sg_{i-1}\cup\sg_i$~($i=1,2,\ldots,k$) are distinct $d$-simplices in $X$.
Here, $k$ is called the \textit{length} of the path.
Two $(d-1)$-simplices $\sg$ and $\sg'$ are said to be \textit{connected in $X$}, denoted by $\sg\xleftrightarrow[]{X}\sg'$, if there exists a path $(\sg_0,\sg_1\ldots,\sg_k)$ such that $\sg_0=\sg$ and $\sg_k=\sg'$.
Given $\sg\in F_{d-1}(\cK_n)$ and $k\in\N$, let $B_k(\sg,X)$ denote the simplicial complex consisting of $\cK_n^{(d-1)}$ and all the $d$-simplices $\tau$ in $X$ such that there exists a path $(\sg_0,\sg_1,\ldots,\sg_i)$ of length at most $k$ satisfying that $\sg_0=\sg$ and $\sg_{i-1}\cup\sg_i=\tau$.
Furthermore, for any $\tau\in F_d(\cK_n)$, we define
\[
B_k(\tau,X)\coloneqq\bigcup_{\sg\colon\text{$(d-1)$-face of $\tau$}}B_k(\sg,X).
\]
Here, a $(d-1)$-face of $\tau$ means a $(d-1)$-simplex contained in $\tau$.
Note that $\tau\in B_k(\tau,X)$, if and only if $\tau\in X$ and $k\ge1$.
For the case of a weighted $d$-complex $\X=(X,\{w_\tau\colon\tau\in F_d(X)\})$, we define weighted $d$-complexes $B_k(\sg,\X)$ and $B_k(\tau,\X)$ as the $d$-complexes $B_k(\sg,X)$ and $B_k(\tau,X)$ with their weights, respectively.
\end{df}

\subsection{Main results}
\label{ssec:mainres}
In this subsection, we state the main general results after introducing some additional conditions regarding randomly weighted complexes.

We recall our randomly weighted $d$-complex model $\X_n$. 
Let $p_n\in(0,1)$, and let $D_n$ be a probability measure on $[0,\infty)$.
Let $B_n=\{b_{n,\tau}\colon\tau\in F_d(\cK_n)\}$ and $W_n=\{w_{n,\tau}\colon\tau\in F_d(\cK_n)\}$ be mutually independent random variables such that $b_{n,\tau}\sim\ber(p_n)$ and $w_{n,\tau}\sim D_n$ for every $\tau\in F_d(\cK_n)$.
We denote by $\X_n= X_{(B_n,W_n)}$ the randomly weighted $d$-complex generated by $(B_n,W_n)$.
\begin{df}
Let $B'_n=\{b'_{n,\tau}\colon\tau\in F_d(\cK_n)\}$ and $W'_n=\{w'_{n,\tau}\colon\tau\in F_d(\cK_n)\}$ be independent copies of $B_n=\{b_{n,\tau}\colon\tau\in F_d(\cK_n)\}$ and $W_n=\{w_{n,\tau}\colon\tau\in F_d(\cK_n)\}$, respectively.
Given a collection $F$ of $d$-simplices in $\cK_n$, let $\X_n^F$ be the randomly weighted $d$-complex obtained by using $\{(b'_{n,\tau},w'_{n,\tau})\colon\tau\in F\}$ in place of $\{(b_{n,\tau},w_{n,\tau})\colon\tau\in F\}$.
When $F=\{\tau\}$, we write $\X_n^\tau$ instead of $\X_n^{\{\tau\}}$ for simplicity.
Note that $\X_n^\emptyset=\X_n$.
For $\tau\in F_d(\cK_n)$ and $F\subset F_d(\cK_n)\setminus\{\tau\}$, we define
\[
\Dl_\tau f(\X_n^F)\coloneqq f(\X_n^F)-f(\X_n^{F\cup\{\tau\}}),
\]
and call it a \textit{randomized derivative of $f(\X_n^F)$ at $\tau$}.
In addition, define for any $k\in\N$, the \textit{local randomized derivative}
\[
\Dl_\tau f(B_k(\tau,\X_n^F))\coloneqq f(B_k(\tau,\X_n^F))-f(B_k(\tau,\X_n^{F\cup\{\tau\}})).
\]
\end{df}
Throughout this article, let us denote 
\[
\lm_n\coloneqq np_n\text{, }\sg_n\coloneqq\sqrt{\var(f(\X_n))}\text{, and }J_n\coloneqq1\vee\E[H(w_{n,\tau},w'_{n,\tau})^6].
\]
Note that $\sigma_n,J_n$ depend on $n$ through $f,H$ as well. Additionally, using arbitrarily fixed disjoint $(d-1)$-simplices $\sg,\sg'\in F_{d-1}(\cK_n)$, we define, 
\[
\gm_{k,n}\coloneqq\P\bigl(\text{$\sg\xleftrightarrow[]{\X_n}\sg'$ by a path of length at most $k$}\bigr),
\] 
for any $k\in\N$. Note that the probability is the same for any two disjoint $(d-1)$-simplices $\sigma,\sigma'$ in $\cK_n$. Furthermore, $C\ge0$ will be constants depending only on $d$, which are always taken to be sufficiently large and can be changed line to line.

The following are the main general results.
\begin{thm}\label{thm:rdm_deriv}
Let $\X_n$ be the randomly weighted $d$-complex as defined above.
Let $f$ be a real-valued function on weighted $d$-complexes in $\cK_n$ that is invariant under weighted isomorphisms and $H$-Lipschitz with $J_n < \infty$.
For each $k\in\N$, let $\dl_{k,n}$ be a constant that dominates
\[
\E[\{\Dl_\tau f(\X_n)-\Dl_\tau f(B_k(\tau,\X_n))\}^2\mid b_{n,\tau},b'_{n,\tau},b_{n,\tau'}]
\]
on the event $\{b_{n,\tau}+b'_{n,\tau}=1\}$ for any disjoint $d$-simplices $\tau,\tau'\in F_d(\cK_n)$.
Furthermore, let $\rho_{k,n}$ be a constant that dominates
\[
\cov(\Dl_\tau f(B_k(\tau,\X_n))\Dl_\tau f(B_k(\tau,\X_n^F)),\Dl_{\tau'} f(B_k(\tau',\X_n))\Dl_{\tau'} f(B_k(\tau',\X_n^{F'}))\mid b_{n,\tau},b'_{n,\tau},b_{n,\tau'},b'_{n,\tau'})
\]
on the event $\{b_{n,\tau}+b'_{n,\tau}=b_{n,\tau'}+b'_{n,\tau'}=1\}$ for any disjoint $d$-simplices $\tau,\tau'\in F_d(\cK_n)$, $F\subset F_d(\cK_n)\setminus\{\tau\}$, and $F'\subset F_d(\cK_n)\setminus\{\tau'\}$.
Then, for every $k\in\N$,
\begin{align}\label{eq:main1}
&d_K\biggl(\cL\biggl(\frac{f(\X_n)-\E[f(\X_n)]}{\sg_n}\biggr),\cN(0,1)\biggr)\nonumber\\
&\le C\biggl(\frac{n^d}{\sg_n^2}\biggr)^{1/2}\biggl[\bigl(J_n^{1/2}\dl_{k,n}^{1/2}+\rho_{k,n}+J_n^{2/3}\gm_{k,n}^{1/2}\bigr)\lm_n^2+J_n^{2/3}\biggl(\frac{\lm_n^2}n+\frac{\lm_n}{n^d}+\frac{\lm_n^3}n\biggr)\biggr]^{1/4}
+\biggl(\frac{n^d}{\sg_n^2}\biggr)^{3/4}\frac{J_n^{1/4}\lm_n^{1/2}}{n^{d/4}}.
\end{align}
Here, $d_K(\mu,\nu)\coloneqq\sup_{t\in\R}|\mu((-\infty,t])-\nu((-\infty,t])|$ is the Kolmogorov distance between probability measures $\mu$ and $\nu$ on $\R$.
\end{thm}
Again observe that $\delta_{k,n},\rho_{k,n}$ are the same for any two disjoint $d$-simplices $\tau,\tau'$ in $\cK_n$.
We can derive the same type of estimate as that in Theorem~\ref{thm:rdm_deriv} in terms of the add-one cost at a $d$-simplex, which is often easy to check in applications.
\begin{df}
For $\tau\in F_d(\cK_n)$ and $F\subset F_d(\cK_n)\setminus\{\tau\}$, let $\X_n^F+\tau$ denote the randomly weighted $d$-complex in $\cK_n$ obtained by adding the $d$-simplex $\tau$ to $\X_n^F$ together with its weight $w_{n,\tau}$; moreover, let $\X_n^F-\tau$ denote the randomly weighted $d$-complex in $\cK_n$ obtained by deleting the $d$-simplex $\tau$ from $\X_n^F$ together with its weight $w_{n,\tau}$.
We then define
\begin{align*}
D_\tau f(\X_n^F)&\coloneqq f(\X_n^F+\tau)-f(\X_n^F-\tau)
\shortintertext{and for any $k\in\N$,}
D_\tau f(B_k(\tau,\X_n^F))&\coloneqq f(B_k(\tau,\X_n^F+\tau))-f(B_k(\tau,\X_n^F-\tau)).
\end{align*}
\end{df}
\begin{thm}\label{thm:add-one}
Let $\X_n$ be the randomly weighted $d$-complex as defined above.
Let $f$ be a real-valued function on weighted $d$-complexes in $\cK_n$ that is invariant under weighted isomorphisms and $H$-Lipschitz with $J_n < \infty$.
For any $k\in\N$ and arbitrarily fixed disjoint $d$-simplices $\tau,\tau'\in F_d(\cK_n)$, define
\begin{align*}
\tilde\dl_{k,n}&\coloneqq\max_{i=0,1}\E[\{D_\tau f(\X_n)-D_\tau f(B_k(\tau,\X_n))\}^2\mid b_{n,\tau'}=i]
\shortintertext{and}
\tilde\rho_{k,n}&\coloneqq\sup_{F,F'\subset F_d(\cK_n)\setminus\{\tau,\tau'\}}\cov(D_\tau f(B_k(\tau,\X_n))D_\tau f(B_k(\tau,\X_n^F)),D_{\tau'} f(B_k(\tau',\X_n))D_{\tau'} f(B_k(\tau',\X_n^{F'}))).
\end{align*}
Then, for every $k\in\N$,
\begin{align}\label{eq:add-one}
&d_K\biggl(\cL\biggl(\frac{f(\X_n)-\E[f(\X_n)]}{\sg_n}\biggr),\cN(0,1)\biggr)\nonumber\\
&\le C\biggl(\frac{n^d}{\sg_n^2}\biggr)^{1/2}\biggl[\bigl(J_n^{1/2}\tilde\dl_{k,n}^{1/2}+\tilde\rho_{k,n}+J_n^{2/3}\gm_{k,n}^{1/3}\bigr)\lm_n^2+J_n^{2/3}\biggl(\frac{\lm_n^2}n+\frac{\lm_n}{n^d}+\frac{\lm_n^3}n\biggr)\biggr]^{1/4}
+\biggl(\frac{n^d}{\sg_n^2}\biggr)^{3/4}\frac{J_n^{1/4}\lm_n^{1/2}}{n^{d/4}}.
\end{align}
\end{thm}
\begin{rem}\label{rem:add-one}
As we will see in Subsection~\ref{ssec:add-one}, we may take $\tilde\dl_{k,n}$ and $\tilde\rho_{k,n}+CJ_n^{2/3}\gm_{k,n}^{1/3}$ as the constants $\dl_{k,n}$ and $\rho_{k,n}$ in Theorem~\ref{thm:rdm_deriv}, respectively.
This is why the term $J_n^{2/3}\gm_{k,n}^{1/2}$ in~\eqref{eq:main1} is replaced into $J_n^{2/3}\gm_{k,n}^{1/3}$ in~\eqref{eq:add-one}.
\end{rem}
By crude estimates on $\gm_{k,n}$ and $\tilde\rho_{k,n}$, we obtain the following corollary.
\begin{cor}\label{cor:no_gm_rho}
Let $d\ge2$, and let $\X_n$ be the randomly weighted $d$-complex as defined above.
Let $f$ be a real-valued function on weighted $d$-complexes in $\cK_n$ that is invariant under weighted isomorphisms and $H$-Lipschitz with $J_n < \infty$.
Let $\tilde\dl_{k,n}$ be the same as in Theorem~\ref{thm:add-one}.
Then, for every $1\le k\le n$,
\begin{align}\label{eq:no_gm_rho}
&d_K\biggl(\cL\biggl(\frac{f(\X_n)-\E[f(\X_n)]}{\sg_n}\biggr),\cN(0,1)\biggr)\nonumber\\
&\le CJ_n^{1/6}(1\vee\lm_n)^{1/2}\biggl(\frac{n^d}{\sg_n^2}\biggr)^{1/2}\biggl[\tilde\dl_{k,n}^{1/8}+\biggl(\frac{k^5(1\vee d\lm_n)^{2k}}n\biggr)^{1/12}\biggr]
+\biggl(\frac{n^d}{\sg_n^2}\biggr)^{3/4}\frac{J_n^{1/4}\lm_n^{1/2}}{n^{d/4}}.
\end{align}
\end{cor}

The following corollary provides useful sufficient conditions to show the CLT for $f(\X_n)$ and in particular, Condition (3) in the below corollary allows for a more robust truncation argument when $\lambda_n \leq 1/d$ for large enough $n$.
\begin{cor}\label{cor:suff_CLT}
Let $d\ge2$, and let $\X_n$ be the randomly weighted $d$-complex as defined above.
Let $f$ be a real-valued function on weighted $d$-complexes in $\cK_n$ that is invariant under weighted isomorphisms and $H$-Lipschitz with $J_n < \infty$.
Let $\tilde\dl_{k,n}$ be the same as in Theorem~\ref{thm:add-one}.
Suppose that $\sg_n^2=\Om(n^d)$.
Furthermore, assume that any of the following conditions hold$:$
\begin{enumerate}
\item $J_n=o(n^\eps)$, $\lm_n=o(n^\eps)$ for any $\eps>0$ and $\tilde\dl_{k,n}=o(n^{-\dl})$ for some $k\in\N$ and $\dl>0;$
\item $J_n=O(1)$, $\lm\coloneqq\limsup_{n\to\infty}\lm_n<\infty$ and there exists a sequence $\{k(n)\}_{n\in\N}\subset\N$ such that
\[
\limsup_{n\to\infty}\frac{k(n)}{\log n}<\frac 1{2\log(1\vee d\lm)}
\quad\text{and}\quad
\lim_{n\to\infty}\tilde\dl_{k(n),n}=0.
\]
\item $J_n = O(1)$, $\lambda_n \leq 1/d$ for all large $n$ and there exists a sequence $\{k(n)\}_{n\in\N}\subset\N$ such that
\[
k(n) = o(n^{1/5})
\quad\text{and}\quad
\lim_{n\to\infty}\tilde\dl_{k(n),n}=0.
\]

\end{enumerate}
Then,
\[
\lim_{n\to\infty}d_K\biggl(\cL\biggl(\frac{f(\X_n)-\E[f(\X_n)]}{\sg_n}\biggr),\cN(0,1)\biggr)=0.
\]
Here we have used the small $o$, the big $O$ and the big $\Omega$ notations ($a_n = \Omega(b_n)$ if $\liminf_{n\to\infty} a_n / b_n > 0$, for two positive sequences $a_n$ and $b_n$).
\end{cor}

\begin{rem}[Variance lower bounds]
For the conclusion of Corollary~\ref{cor:suff_CLT}, we need to check the assumption $\sg_n^2=\Om(n^d)$.
Applying \cite[Corollary~2.4]{LP17} to our setting, we have the following lower bound on $\sg_n^2$.
\begin{align*}
\sg_n^2=\var(f(\X_n))
&\ge\sum_{\tau\in F_d(\cK_n)}\E[\E[\Dl_\tau f(\X_n)\mid(B'_n,W'_n)]^2]\\
&=\sum_{\tau\in F_d(\cK_n)}\E[\E[\Dl_\tau f(\X_n)\mid(b'_{n,\tau},w'_{n,\tau})]^2]\\
&=\binom n{d+1}\E[\E[\1_{\{b_{n,\tau}+b'_{n,\tau} \geq 1\}}\Dl_\tau f(\X_n)\mid(b'_{n,\tau},w'_{n,\tau})]^2] \\
& \geq \binom n{d+1}\E[\E[\1_{\{b_{n,\tau}+b'_{n,\tau} = 1\}}\Dl_\tau f(\X_n)\mid(b'_{n,\tau},w'_{n,\tau})]^2].
\end{align*}
The above lower bound can be further simplified in the case of unweighted $d$-complexes in $\cK_n$, i.e., when the statistic does not depend on $w_{n,\tau}$ or equivalently $w_{\cdot,\cdot} \equiv 1$. In this case, by using that $b'_{n,\tau} \in \{0,1\}$, we derive that
\begin{align*}
\E[\1_{\{b_{n,\tau}+b'_{n,\tau} = 1\}}\Dl_\tau f(\X_n)\mid b'_{n,\tau}]^2
&=\sum_{i=0,1}\E[\1_{\{b_{n,\tau}+b'_{n,\tau} = 1\}}\Dl_\tau f(\X_n)\mid b'_{n,\tau}=i]^2\1_{\{b'_{n,\tau}=i\}}\\
&=\sum_{i=0,1}\E[\1_{\{b_{n,\tau}=1-i\}}\Dl_\tau f(\X_n)\mid b'_{n,\tau}=i]^2\1_{\{b'_{n,\tau}=i\}}\\
&=\sum_{i=0,1}\E[\1_{\{b_{n,\tau}=1-i\}}D_\tau f(\X_n)\mid b'_{n,\tau}=i]^2\1_{\{b'_{n,\tau}=i\}}\\
&=\sum_{i=0,1}\P(b_{n,\tau}=1-i)\E[D_\tau f(\X_n)]^2\1_{\{b'_{n,\tau}=i\}}
\end{align*}
and so
\begin{equation}
\label{eq:var_lb_unweight}
\sigma_n^2 \geq 2\binom n{d+1} p_n(1-p_n)\E[D_{\tau}f(\X_n)]^2.
\end{equation}
\end{rem}
\begin{rem}[Variance upper bounds]
On the other hand, we also note that the classical Efron--Stein inequality (see, e.g.,~\cite[Theorem~3.1]{BLM13}) yields the following upper bound on $\sg_n^2$.
\begin{align*}
\sg_n^2=\var(f(\X_n))
&\le\frac12\sum_{\tau\in F_d(\cK_n)}\E[\Dl_\tau f(\X_n)^2]\\
&\le\frac12\sum_{\tau\in F_d(\cK_n)}\E[\1_{\{b_{n,\tau} \vee b'_{n,\tau}=1\}}H(w_\tau,w'_\tau)^2]\\
&\le\frac12\sum_{\tau\in F_d(\cK_n)}\P(b_{n,\tau} \vee b'_{n,\tau}=1)\E[H(w_\tau,w'_\tau)^2]\\
&\le p_n\sum_{\tau\in F_d(\cK_n)}\E[H(w_\tau,w'_\tau)^2]\\
&\le n^d\lm_nJ_n^{1/3}.
\end{align*}
For the last line, we used H\"older's inequality.
\end{rem}
%
\section{Applications}
\label{sec:applns}

In this section, we consider some combinatorial optimization problems on random simplicial complexes, and apply the main results in Section~\ref{sec:model-result} to obtain their CLTs. In particular, we prove a CLT for nearest face-weights (Theorem~\ref{thm:CLT_NN}) in Section~\ref{sec:nearfaceweight} and a normal approximation for local statistics in Section~\ref{sec:localstat}.

\subsection{Nearest neighbor weights - Proof of Theorem~\ref{thm:CLT_NN}}
\label{sec:nearfaceweight}
In this section, we shall prove one of our main applications, the CLT for nearest neighbour weights - Theorem~\ref{thm:CLT_NN}. Recall the definition of the nearest face-weight $\text{NN}_n(\sigma)$ and total nearest face-weight $\text{NN}_n(\K_n)$ from \eqref{e:nearfaceweight} and \eqref{e:totnearfaceweight}, respectively. We now give some basic observations before going to the proof.
We cannot apply Corollary~\ref{cor:suff_CLT} to $\text{NN}_n(\K_n)$ directly because $\lm_n=n$ in this setting.
Hence, we use a truncation argument.

Let $\a>0$ be fixed.
For any $\sg\in F_{d-1}(\cK_n)$, we set
\[
\text{NN}_n^\a(\sg)
\coloneqq\min_{\sg\subset\tau\in F_d(\cK_n)}(w_{n,\tau}\wedge\a) = \text{NN}_n(\sigma) \wedge \alpha.
\]
Here, $a\wedge b$ denotes $\min\{a,b\}$ for any $a,b\in\R$. 
We then define the \textit{$\a$-diluted nearest neighbor weight} $\text{NN}_n^\a(\K_n)$ of $\K_n$ by
\[
\text{NN}_n^\a(\K_n)\coloneqq\sum_{\sg\in F_{d-1}(\cK_n)}\text{NN}_n^\a(\sg).
\]
Clearly, we have $\text{NN}_n^\a(\K_n)\nearrow\text{NN}_n(\K_n)$ as $\a\to\infty$.
Let $K_n^\a$ be the $d$-complex in $\cK_n$ consisting of $\cK_n^{(d-1)}$ and all $d$-simplices whose weights are at most $\a$.
Additionally, let $\K_n^\a$ denote the weighted $d$-complex $(K_n^\a,\{w_{n,\tau}\colon\tau\in F_d(K_n^\a)\})$ in $\cK_n$.
Note that $\text{NN}_n^\a(\K_n)$ is a function of $\K_n^\a$.
Note that for any $\sg\in F_{d-1}(\cK_n)$,
\[
\text{NN}_n^\a(\sg) = \begin{cases}
\min_{\sg\subset\tau\in F_d(K_n^\a)}(w_{n,\tau}\wedge\a)    &\text{if $\deg_{K_n^\a}(\sg)>0$,}\\
\a                         &\text{otherwise.}
\end{cases}
\]
Hence, given a weighted $d$-complex $\X=(X,\{w_\tau\colon\tau\in F_d(X)\})$ in $\cK_n$ and $\sg\in F_{d-1}(\cK_n)$, we set
\[
f^\a(\X,\sg)\coloneqq\begin{cases}
\min_{\sg\subset\tau\in F_d(X)}(w_\tau\wedge\a)    &\text{if $\deg_{X}(\sg)>0$,}\\
\a                         &\text{otherwise,}
\end{cases}
\]
and define
\[
f^\a(\X)\coloneqq\sum_{\sg\in F_{d-1}(\cK_n)}f^\a(\X,\sg).
\]
Then, $\text{NN}_n^\a(\K_n)=f^\a(\K_n^\a)$.
Furthermore, $f^\a$ is Lipschitz.
Indeed, for any $B=\{b_\tau\colon\tau\in F_d(\cK_n)\}$, $W=\{w_\tau\colon\tau\in F_d(\cK_n)\}$, $B'=\{b'_\tau\colon\tau\in F_d(\cK_n)\}$, and $W'=\{w'_\tau\colon\tau\in F_d(\cK_n)\}$,
\begin{equation}\label{eq:Lip_NNlm}
\bigl|f^\a\bigl(\X_{(B,W)}\bigr)-f^\a\bigl(\X_{(B',W')}\bigr)\bigr|
\le\sum_{\substack{\tau\colon b_\tau\vee b'_\tau=1\\b_\tau\neq b'_\tau}}(d+1)\a.
\end{equation}

Now, using a fixed $\tau\in F_d(\cK_n)$, we set
\[
p_{n,\a} \coloneqq\P(w_{n,\tau}\le\a)=1-e^{-\a/n}
\quad\text{and}\quad
D_n^\a\coloneqq\P(w_{n,\tau}\in\cdot\mid w_{n,\tau}\le\a).
\]
Let $B_n^\a=\{b_{n,\tau}^\a\colon\tau\in F_d(\cK_n)\}$ and $W_n^\a=\{w_{n,\tau}^\a\colon\tau\in F_d(\cK_n)\}$ be mutually independent random variables such that $b_{n,\tau}^\a\sim\ber(p_{n,\a})$ and $w_{n,\tau}^\a\sim D_n^\a$ for every $\tau\in F_d(\cK_n)$.
We denote by $\X_n^\a$ the randomly weighted $d$-complex generated by $(B_n^\a,W_n^\a)$.
Then, one can easily verify that $\X_n^\a$ and $\K_n^\a$ have the same distribution.
In particular, $\text{NN}_n^\a(\K_n)=f^\a(\K_n^\a)$ has the same distribution as that of $f^\a(\X_n^\a)$.

With the above observation, we will show the CLT for $\text{NN}_n^{\a_n}(\K_n)$ with a sufficiently large $\a_n$, tending to infinity as $n\to\infty$, by applying Corollary~\ref{cor:suff_CLT}.
The following lemma will then help us lift the result into Theorem~\ref{thm:CLT_NN}. 
\begin{lem}\label{lem:NN-NNalpha}
Let $C_2>0$ be fixed.
Then, there exists a constant $C_1>0$ such that for $\a_n=C_1\log n$ and sufficiently large $n$,
\[
\P(\text{NN}_n(\K_n)=\text{NN}_n^{\a_n}(\K_n))\ge1-3n^{-C_2}.
\]
\end{lem}
\begin{proof}
We first note that $\text{NN}_n(\K_n)\neq\text{NN}_n^{\a_n}(\K_n)$ implies that there exists a $\sg\in F_{d-1}(\cK_n)$ such that $\deg_{K_n^{\a_n}}(\sg)=0$.
Therefore,
\[
\P(\text{NN}_n(\K_n)\neq\text{NN}_n^{\a_n}(\K_n))
\le\P\biggl(\bigcup_{\sg\in F_{d-1}(\cK_n)}\{\deg_{K_n^{\a_n}}(\sg)=0\}\biggr)
\le\binom nd\P(\deg_{K_n^{\a_n}}([d])=0).
\]
Furthermore, since $\deg_{K_n^{\a_n}}([d])\sim\bin(n-d,p_{n,\a_n})$, the Chernoff bound (e.g.~\cite[Theorem~12.6]{BHK20}) yields
\begin{equation}\label{eq:Chernoff}
\P(\deg_{K_n^{\a_n}}([d])=0)
\le\P(\deg_{K_n^{\a_n}}([d])\le(n-d)p_{n,\a_n}/2)
\le3e^{-(n-d)p_{n,\a_n}/32}.
\end{equation}
Now, we set $C_1\coloneqq 64(C_2+d)$.
Then, since $p_{n,\a_n}=1-e^{-\a_n/n}\ge\a_n/\{2(n-d)\}$ for sufficiently large $n$, the right-hand side of~\eqref{eq:Chernoff} is bounded above by $3n^{-(C_2+d)}$.
The conclusion follows by combining the above estimates.
\end{proof}
Now, we state and prove variance asymptotics.
\begin{lem}\label{lem:var_NN}
It holds that
\[
\var(\text{NN}_n(\K_n))\sim \Big( 1+ \frac{d}{2}\Big)\binom nd.
\]
\end{lem}
\begin{proof}
Noting that
\[
\text{NN}_n(\K_n)=\sum_{\sg\in F_{d-1}(\cK_n)}\text{NN}_n(\sg),
\]
we can write
\[
\var(\text{NN}_n(\K_n))
=\sum_{\sg\in F_{d-1}(\cK_n)}\var(\text{NN}_n(\sg))+\sum_{\substack{\sg,\sg'\in F_{d-1}(\cK_n)\\\dim(\sg\cup\sg')=d}}\cov(\text{NN}_n(\sg),\text{NN}_n(\sg')).
\]
Here, $\text{NN}_n(\sg)\sim\text{Exp}((n-d)/n)$ since $\text{NN}_n(\sg)$ is the minimum of $(n-d)$ i.i.d.\ random variables sampled from $\text{Exp}(1/n)$, which implies that
\[
\E[\text{NN}_n(\sg)]=\frac n{n-d}
\quad\text{and}\quad
\var(\text{NN}_n(\sg))=\biggl(\frac n{n-d}\biggr)^2.
\]
Suppose that we prove that for fixed $\sg,\sg'\in F_{d-1}(\cK_n)$ such that $\dim(\sg\cup\sg')=d$,
\begin{equation}\label{eq:cov_NN}
\cov(\text{NN}_n(\sg),\text{NN}_n(\sg'))\sim\frac1{2n}.
\end{equation}
This immediately gives that
$$ \sum_{\substack{\sg,\sg'\in F_{d-1}(\cK_n)\\\dim(\sg\cup\sg')=d}}\cov(\text{NN}_n(\sg),\text{NN}_n(\sg')) \sim \binom{n}{d}d(n-d)\frac{1}{2n},$$
which can be substituted in the above variance identity to complete the proof. Now, define
\begin{align*}
\widetilde{\text{NN}}_n(\sg)\coloneqq\min\{w_{n,\tau}\mid\sg\subset\tau\in F_d(\cK_n)\setminus\{\sg\cup\sg'\}\}
\shortintertext{and}
\widetilde{\text{NN}}_n(\sg')\coloneqq\min\{w_{n,\tau}\mid\sg'\subset\tau\in F_d(\cK_n)\setminus\{\sg\cup\sg'\}\}.
\end{align*}
Note that $\widetilde{\text{NN}}_n(\sg)$ and $\widetilde{\text{NN}}_n(\sg')$ are independent.
Furthermore, we have
\[
\text{NN}_n(\sg)=\widetilde{\text{NN}}_n(\sg)\wedge w_{n,\sg\cup\sg'}
\quad\text{and}\quad
\text{NN}_n(\sg')=\widetilde{\text{NN}}_n(\sg')\wedge w_{n,\sg\cup\sg'}.
\]
Therefore, from the conditional independence of $\text{NN}_n(\sg)$ and $\text{NN}_n(\sg')$ given $w_{n,\sg\cup\sg'}$,
\begin{align*}
&\E[\text{NN}_n(\sg) \times \text{NN}_n(\sg')]\\
&=\E\bigl[\E[(\widetilde{\text{NN}}_n(\sg)\wedge w_{n,\sg\cup\sg'})(\widetilde{\text{NN}}_n(\sg')\wedge w_{n,\sg\cup\sg'})\mid w_{n,\sg\cup\sg'}]\bigr]\\
&=\E\bigl[\E[\widetilde{\text{NN}}_n(\sg)\wedge w_{n,\sg\cup\sg'}\mid w_{n,\sg\cup\sg'}]\E[\widetilde{\text{NN}}_n(\sg')\wedge w_{n,\sg\cup\sg'}\mid w_{n,\sg\cup\sg'}]\bigr].
\end{align*}
Since $\widetilde{\text{NN}}_n(\sg)\sim\text{Exp}((n-d-1)/n)$, a simple calculation yields
\begin{align*}
\E\bigl[\widetilde{\text{NN}}_n(\sg)\wedge w_{n,\sg\cup\sg'}\mid w_{n,\sg\cup\sg'}\bigr]
&=\int_0^\infty(x\wedge w_{n,\sg\cup\sg'})\frac{n-d-1}n\exp\biggl(-\frac{n-d-1}nx\biggr)\,dx\\
&=\frac n{n-d-1}\biggl(1-\exp\biggl(-\frac{n-d-1}nw_{n,\sg\cup\sg'}\biggr)\biggr).
\end{align*} 
A similar equation holds for $\E\bigl[\widetilde{\text{NN}}_n(\sg')\wedge w_{n,\sg\cup\sg'}\mid w_{n,\sg\cup\sg'}\bigr]$.
Combining the above equations, we obtain
\begin{align}\label{eq:ENN}
&\E[\text{NN}_n(\sg)\text{NN}_n(\sg')]\nonumber\\
&=\biggl(\frac n{n-d-1}\biggr)^2\E\biggl[\biggl(1-\exp\biggl(-\frac{n-d-1}nw_{n,\sg\cup\sg'}\biggr)\biggr)^2\biggr]\nonumber\\
&=\biggl(\frac n{n-d-1}\biggr)^2\biggl(1-2\E\biggl[\exp\biggl(-\frac{n-d-1}nw_{n,\sg\cup\sg'}\biggr)\biggr]+\E\biggl[\exp\biggl(-2\frac{n-d-1}nw_{n,\sg\cup\sg'}\biggr)\biggr]\biggr)\nonumber\\
&=\biggl(\frac n{n-d-1}\biggr)^2\biggl(1-2\frac{1/n}{1/n+(n-d-1)/n}+\frac{1/n}{1/n+2(n-d-1)/n}\biggr)\nonumber\\
&=1+\frac{2d+1/2}n+O\biggl(\frac1{n^2}\biggr).
\end{align}
On the other hand,
\begin{equation}\label{eq:ENEN}
\E[\text{NN}_n(\sg)]\E[\text{NN}_n(\sg')]
=\biggl(\frac n{n-d}\biggr)^2
=1+\frac{2d}n+O\biggl(\frac1{n^2}\biggr).
\end{equation}
Thus,~\eqref{eq:cov_NN} follows from~\eqref{eq:ENN} and~\eqref{eq:ENEN}, which completes the proof.
\end{proof}
From Lemmas~\ref{lem:NN-NNalpha} and~\ref{lem:var_NN}, we have the following.
\begin{lem}\label{lem:varNN-varNNlm}
There exists a constant $C_1>0$ such that for $\a_n=C_1\log n$,
\[
\lim_{n\to\infty}\frac{\var(\text{NN}_n^{\a_n}(\K_n))}{\var(\text{NN}_n(\K_n))}=1.
\]
\end{lem}
\begin{proof}
Since
\[
|\var(\text{NN}_n^{\a_n}(\K_n))^{1/2}-\var(\text{NN}_n(\K_n))^{1/2}|
\le\var(\text{NN}_n(\K_n)-\text{NN}_n^{\a_n}(\K_n))^{1/2},
\]
we have
\[
\biggl|\frac{\var(\text{NN}_n^{\a_n}(\K_n))^{1/2}}{\var(\text{NN}_n(\K_n))^{1/2}}-1\biggr|
\le\biggl(\frac{\var(\text{NN}_n(\K_n)-\text{NN}_n^{\a_n}(\K_n))}{\var(\text{NN}_n(\K_n))}\biggr)^{1/2}.
\]
In order to show the right-hand side of the above equation converges to zero, we set $C_2 = 2(d+1)$.
From Lemma~\ref{lem:NN-NNalpha}, we can take $C_1>0$ such that for $\a_n=C_1\log n$ and sufficiently large $n$,
\[
\P(\text{NN}_n(\K_n)=\text{NN}_n^{\a_n}(\K_n))\ge1-3n^{-C_2}.
\]
Therefore, using the Cauchy--Schwarz inequality, we have
\begin{align*}
\var(\text{NN}_n(\K_n)-\text{NN}_n^{\a_n}(\K_n))
&\le\E[(\text{NN}_n(\K_n)-\text{NN}_n^{\a_n}(\K_n))^2]\\
&\le\E[(\text{NN}_n(\K_n)-\text{NN}_n^{\a_n}(\K_n))^4]^{1/2}\P(\text{NN}_n(\K_n)\neq\text{NN}_n^{\a_n}(\K_n))^{1/2}\\
&\le\E[\text{NN}_n(\K_n)^4]^{1/2}\cdot\sqrt3n^{-C_2/2}.
\end{align*}
Noting that
\[
\text{NN}_n(\K_n)=\sum_{\sg\in F_{d-1}(\cK_n)}\text{NN}_n(\sg),
\]
and that each $\text{NN}_n(\sg)$ obeys $\text{Exp}((n-d)/n)$, we have
\[
\E[\text{NN}_n(\K_n)^4]^{1/4}
\le\binom nd\E[\text{NN}_n([d])^4]^{1/4}
\le\binom nd\biggl(\frac{24}{\{(n-d)/n\}^4}\biggr)^{1/4}
\le\frac3{d!}n^d,
\]
for sufficiently large $n$.
Combining the above equations yields
\[
\var(\text{NN}_n(\K_n)-\text{NN}_n^{\a_n}(\K_n))
\le Cn^{2d-C_2/2}
= Cn^{d-1}.
\]
Thus, the conclusion follows from Lemma~\ref{lem:var_NN}.
\end{proof}

Now, we are ready to show the CLT for $\text{NN}_n^{\a_n}(\K_n)$ with a sufficiently large $\a_n$, tending to infinity as $n\to\infty$, by applying Corollary~\ref{cor:suff_CLT}.
\begin{lem}\label{lem:CLT_NNlm}
There exists a constant $C_1>0$ such that for $\a_n=C_1\log n$,
\[
\lim_{n\to\infty}d_K\biggl(\cL\biggl(\frac{\text{NN}_n^{\a_n}(\K_n)-\E[\text{NN}_n^{\a_n}(\K_n)]}{\sqrt{\var(\text{NN}_n^{\a_n}(\K_n))}}\biggr),\cN(0,1)\biggr)=0.
\]
\end{lem}
\begin{proof}
We take a constant $C_1>0$ given in Lemma~\ref{lem:varNN-varNNlm}.
Since $\text{NN}_n^{\a_n}(\K_n)$ has the same distribution of that of $f^{\a_n}(\X_n^{\a_n})$, it suffices to prove that
\[
\lim_{n\to\infty}d_K\biggl(\cL\biggl(\frac{f^{\a_n}(\X_n^{\a_n})-\E[f^{\a_n}(\X_n^{\a_n})]}{\sqrt{\var(f^{\a_n}(\X_n^{\a_n}))}}\biggr),\cN(0,1)\biggr)=0
\]
by verifying the assumptions of Corollary~\ref{cor:suff_CLT}.
We first note that $J_n\le\{(d+1)\a_n\}^6$ from~\eqref{eq:Lip_NNlm} and that $\sg_n^2\coloneqq\var(f^{\a_n}(\X_n^{\a_n}))=\var(\text{NN}_n^{\a_n}(\K_n))=\Om(n^d)$ from Lemmas~\ref{lem:var_NN} and~\ref{lem:varNN-varNNlm}.
Additionally, $np_{n,\a_n}=n(1-e^{-\a_n/n})\le\a_n=O(\log n)$. Thus, apart from deriving suitable bounds on $\tilde\dl_{k,n}$,  we have checked the assumptions $J_n = o(n^{\epsilon})$ and $\lambda_n = o(n^{\epsilon})$ for all $\epsilon$ in Corollary~\ref{cor:suff_CLT}(1).

We only need to check now the condition on $\tilde\dl_{k,n}$. Recall that
\[
\tilde\dl_{k,n}\coloneqq\max_{i=0,1}\E[\{D_\tau f^{\a_n}(\X_n)-D_\tau f^{\a_n}(B_k(\tau,\X_n))\}^2\mid b_{n,\tau'}=i],
\]
where $\tau,\tau'$ are disjoint $d$-simplices in $\cK_n$.
From the definition of $f^{\a_n}$, we have
\begin{align*}
D_\tau f^{\a_n}(\X_n)
&=f^{\a_n}(\X_n+\tau)-f^{\a_n}(\X_n-\tau)\\
&=\sum_{\sg\in F_{d-1}(\cK_n)}\{f^{\a_n}(\X_n+\tau,\sg)-f^{\a_n}(\X_n-\tau,\sg)\}\\
&=\sum_{\sg\colon\text{$(d-1)$-face of $\tau$}}\{f^{\a_n}(\X_n+\tau,\sg)-f^{\a_n}(\X_n-\tau,\sg)\}.
\end{align*}
Similarly, whenever $k\ge1$,
\begin{align*}
D_\tau f^{\a_n}(B_k(\tau,\X_n))
&=f^{\a_n}(B_k(\tau,\X_n+\tau))-f^{\a_n}(B_k(\tau,\X_n-\tau))\nonumber\\
&=\sum_{\sg\in F_{d-1}(\cK_n)}\{f^{\a_n}(B_k(\tau,\X_n+\tau),\sg)-f^{\a_n}(B_k(\tau,\X_n-\tau),\sg)\}\nonumber\\
&=\sum_{\sg\colon\text{$(d-1)$-face of $\tau$}}\{f^{\a_n}(B_k(\tau,\X_n+\tau),\sg)-f^{\a_n}(B_k(\tau,\X_n-\tau),\sg)\}\nonumber\\
&=\sum_{\sg\colon\text{$(d-1)$-face of $\tau$}}\{f^{\a_n}(\X_n+\tau,\sg)-f^{\a_n}(\X_n-\tau,\sg)\}.
\end{align*}
Therefore, $\tilde\dl_{k,n}=0$ for $k\ge1$. Consequently, we can apply Corollary~\ref{cor:suff_CLT}(1) to $f^{\a_n}(\X_n^{\a_n})$ to obtain the conclusion.
\end{proof}

Using Lemmas~\ref{lem:NN-NNalpha}--\ref{lem:CLT_NNlm}, we can prove Theorem~\ref{thm:CLT_NN}.
\begin{proof}[Proof of Theorem~\ref{thm:CLT_NN}]
We can write
\begin{align}\label{eq:CLT_NN}
\frac{\E[\text{NN}_n(\K_n)]-\E[\text{NN}_n(\K_n)]}{\sqrt{\var(\text{NN}_n(\K_n))}}
&=\biggl(\frac{\var(\text{NN}_n^{\a_n}(\K_n))}{\var(\text{NN}_n(\K_n))}\biggr)^{1/2}\frac{\text{NN}_n^{\a_n}(\K_n)-\E[\text{NN}_n^{\a_n}(\K_n)]}{\sqrt{\var(\text{NN}_n^{\a_n}(\K_n))}}\nonumber\\
&\qad+\frac{\text{NN}_n(\K_n)-\text{NN}_n^{\a_n}(\K_n)}{\sqrt{\var(\text{NN}_n(\K_n))}}\nonumber\\
&\qad-\frac{\E[\text{NN}_n(\K_n)]-\E[\text{NN}_n^{\a_n}(\K_n)]}{\sqrt{\var(\text{NN}_n(\K_n))}}.
\end{align}
Let $C_2=d+2$.
From Lemmas~\ref{lem:CLT_NNlm} and~\ref{lem:NN-NNalpha}, we can take $C_1>0$ such that for $\a_n=C_1\log n$,
\begin{equation}\label{eq:CLT_NN1}
\frac{\text{NN}_n^{\a_n}(\K_n)-\E[\text{NN}_n^{\a_n}(\K_n)]}{\sqrt{\var(\text{NN}_n^{\a_n}(\K_n))}}\xrightarrow[n\to\infty]{d}\cN(0,1)
\end{equation}
and $\P(\text{NN}_n(\K_n)=\text{NN}_n^{\a_n}(\K_n))\ge1-3n^{-C_2}$ holds for sufficiently large $n$.
Then, in a similar way to that in the proof of Lemma~\ref{lem:varNN-varNNlm}, we have
\begin{align*}
\E[|\text{NN}_n(\K_n)-\text{NN}_n^{\a_n}(\K_n)|]
&\le\E[(\text{NN}_n(\K_n)-\text{NN}_n^{\a_n}(\K_n))^2]^{1/2}\P(\text{NN}_n(\K_n)\neq\text{NN}_n^{\a_n}(\K_n))^{1/2}\\
&\le\E[\text{NN}_n(\K_n)^2]^{1/2}\cdot\sqrt3n^{-C_2/2}\\
&\le Cn^{d-C_2/2}\\
&= Cn^{d/2-1}.
\end{align*}
Combining the above estimate with Lemma~\ref{lem:var_NN}, the second and third terms in the right-hand side of~\eqref{eq:CLT_NN} converge to zero in distribution.
From~\eqref{eq:CLT_NN1} and again Lemma~\ref{lem:var_NN}, the first term in the right-hand side of~\eqref{eq:CLT_NN} converges to $\cN(0,1)$ in distribution, which completes the proof.
\end{proof}

\subsection{Local statistics}
\label{sec:localstat}
In this subsection, we investigate local statistics of randomly weighted $d$-complexes $\X_n$. We first state and prove the main theorem and then discuss examples of local statistics. In order to define our statistic, we begin by defining a variant of complex consisting of simplices in the `$k$-neighbourhood' of a simplex; see Definition \ref{d:knbhd_complex}.
\begin{df}\label{df:M-ball}
Given a $d$-complex $X$ in $\cK_n$, $\sg\in F_{d-1}(\cK_n)$, and $M\in\N$, let $(X,\sg)_M$ denote the simplicial complex generated by $\sg$ and all the $d$-simplices $\tau$ in $X$ such that there exists a path $(\sg_0,\sg_1,\ldots,\sg_i)$ of length at most $M$ satisfying that $\sg_0=\sg$ and $\sg_{i-1}\cup\sg_i=\tau$.
Furthermore, for any $\tau\in F_d(\cK_n)$, we define
\[
(X,\tau)_M\coloneqq\bigcup_{\sg\colon\text{$(d-1)$-face of $\tau$}}(X,\sg)_M.
\]
For the case of a weighted $d$-complex $\X$, we set $(\X,\sg)_M$ and $(\X,\tau)_M$ as $(X,\sg)_M$ and $(X,\tau)_M$ with their weights, respectively.
\end{df}

The difference with Definition \ref{d:knbhd_complex} is that the former considers also the $(d-1)$-simplices outside the `$k$-neighbourhood' and these are not part of any $d$-simplex, i.e., isolated. Now, let $g$ be a real-valued function that takes finite simplicial complexes with weights on its $d$-simplices as inputs.
We assume that $g$ is invariant under weighted isomorphism.
For each fixed $M\in\N$, we consider the following real-valued function of weighted $d$-complexes in $\cK_n$:
\begin{equation}\label{df:local_statistic}
f(\X) = f^M(\X):=\sum_{\sg\in F_{d-1}(\cK_n)}g((\X,\sg)_M).
\end{equation}
%
The following is the main result in this subsection.
Recall that $\X_n$ is the randomly weighted $d$-complex generated by $(B_n,W_n)$.
\begin{thm}\label{thm:CLT_M-bounded}
Let $d\ge2$, and let $f$ be a real-valued function on weighted $d$-complexes in $\cK_n$ defined in~\eqref{df:local_statistic}.
Furthermore, assume that $f$ is $H$-Lipschitz with $J_n < \infty$.
Then, for every $2M\le k \le n$,
\begin{align*}
&d_K\biggl(\cL\biggl(\frac{f(\X_n)-\E[f(\X_n)]}{\sg_n}\biggr),\cN(0,1)\biggr)\\
&\le CJ_n^{1/6}(1\vee\lm_n)^{1/2}\biggl(\frac{n^d}{\sg_n^2}\biggr)^{1/2}\biggl(\frac{k^5(1\vee d\lm_n)^{2k}}n\biggr)^{1/12}+\biggl(\frac{n^d}{\sg_n^2}\biggr)^{3/4}\frac{J_n^{1/4}\lm_n^{1/2}}{n^{d/4}}.
\end{align*}
\end{thm}
\begin{proof}
Let $\tau,\tau'\in F_d(\cK_n)$ be disjoint $d$-simplices.
We consider
\[
\tilde\dl_{k,n}\coloneqq\max_{i=0,1}\E[\{D_\tau f(\X_n)-D_\tau f(B_k(\tau,\X_n))\}^2\mid b_{n,\tau'}=i].
\]
From~\eqref{df:local_statistic}, we have
\begin{align*}
D_\tau f(\X_n)
&=f(\X_n+\tau)-f(\X_n-\tau)\nonumber\\
&=\sum_{\sg\in F_{d-1}(\cK_n)}\{g((\X_n+\tau,\sg)_M)-g((\X_n-\tau,\sg)_M)\}\nonumber\\
&=\sum_{\sg\in F_{d-1}((\X_n,\tau)_M)}\{g((\X_n+\tau,\sg)_M)-g((\X_n-\tau,\sg)_M)\}.
\end{align*}
Similarly, whenever $k\ge2M$,
\begin{align*}
D_\tau f(B_k(\tau,\X_n))
&=f(B_k(\tau,\X_n+\tau))-f(B_k(\tau,\X_n-\tau))\\
&=\sum_{\sg\in F_{d-1}(\cK_n)}\{g((B_k(\tau,\X_n+\tau),\sg)_M)-g((B_k(\tau,\X_n-\tau),\sg)_M)\}\\
&=\sum_{\sg\in F_{d-1}((\X_n,\tau)_M)}\{g((B_k(\tau,\X_n+\tau),\sg)_M)-g((B_k(\tau,\X_n-\tau),\sg)_M)\}\\
&=\sum_{\sg\in F_{d-1}((\X_n,\tau)_M)}\{g((\X_n+\tau,\sg)_M)-g((\X_n-\tau,\sg)_M)\}.
\end{align*}
For the last line, we note that if $k\ge2M$, then $(\X_n\pm\tau,\sg)_M$ are contained in $B_k(\tau,\X_n\pm\tau)$ for each $\sg\in F_{d-1}((\X_n,\tau)_M)$, hence $(B_k(\tau,\X_n\pm\tau),\sg)_M=(\X_n\pm\tau,\sg)_M$.
Therefore, $\tilde\dl_{k,n}=0$ for any $k\ge2M$.
Thus, the conclusion follows immediately from Corollary~\ref{cor:no_gm_rho}.
\end{proof}


We provide two examples of local statistics of $d$-Linial-Meshulam complex, namely isolated simplex count and local Betti number. Though one can construct more such examples,  the verification of the variance lower bound needs to be done in each case separately. For isolated simplex count,  we borrow variance lower bounds from \cite{ERTZ21} and obtain CLT for isolated simplex count easily from our theorem above albeit with poorer rates.  For the case of local Betti number count, we shall derive the necessary variance lower bound using~\eqref{eq:var_lb_unweight}.

Let $\lm>0$ be fixed, and let $Y_n$ be a $d$-Linial--Meshulam complex on $[n]$ with parameter $\lm/n$, i.e., a random subcomplex of $\cK_n$ given by
\[
\P(Y_n=Y)=\begin{cases}
(\lm/n)^{f_d(Y)}(1-\lm/n)^{\binom n{d+1}-f_d(Y)}	&\text{if $\cK_n^{(d-1)}\subset Y\subset\cK_n^{(d)}$,}\\
0	&\text{otherwise.}
\end{cases}
\]
The random complex $Y_n$ can be constructed by taking all possible $(d-1)$-faces on $n$ vertices and then adding $d$-faces independently with probability $\lm/n$. Note that $1$-Linial--Meshulam complex can be regarded as an \ER graph.
\begin{eg}(Number of Isolated simplices)
\label{eg:isolface}
Define a function $g(\X)$ on weighted complexes to be $1$, if $\X$ has no $d$-simplex, and $0$ otherwise. Then $g((\X,\sg)_1) = 1$, if $\sigma$ is an isolated $(d-1)$-simplex in $X$, and $g((\X,\sg)_1) = 0$ otherwise. Thus, we have that  $f(\X) = f^1(\X)$ is the number of isolated simplices in $\X$. In this case, it is easy to check that $\sup_{n \geq 1}J_n < \infty$ and also it is shown in \cite[Section 8]{ERTZ21} that $\sigma_n^2 = \Omega(n^d)$. Thus from Theorem \ref{thm:CLT_M-bounded}, we obtain
$$ d_K\biggl( \cL\biggl(\frac{f(\X_n)-\E[f(\X_n)]}{\sg_n}\biggr) , \cN(0,1)\biggr) = O(n^{-1/12}).$$
This is weaker than the rate $n^{-d/2}$ obtained in \cite[Theorem 1.4]{ERTZ21}. As discussed before, the latter article uses discrete second-order Poincar\'{e} inequality that is specific to the unweighted case and involves computation of iterated add-one cost functions. This is reflective of the differences between our results and those of \cite{ERTZ21}. The latter approach will give better bounds in the unweighted case while ours applies to more general statistics but will possibly give sub-optimal bounds. A similar difference is noted for the Poisson case in \cite[Remark 1.2]{LPY20}.
\end{eg}
Next, we consider an example from a topological aspect, where we again do not consider weights (or just set $w_{n,\tau}\equiv1$). 
A simplicial complex $X$ is said to be \textit{pure $d$-dimensional} if for any $\sg\in X$, there exists a $d$-simplex $\tau$ in $X$ containing $\sg$.
A pure $d$-dimensional simplicial complex $X$ is said to be \textit{strongly connected} if $\sg\xleftrightarrow[]{X}\sg'$ for any $\sg,\sg'\in F_{d-1}(X)$.
For convenience, we also regard a simplicial complex generated by a single $(d-1)$-simplex as a strongly connected simplicial complex.
%
A \textit{strongly connected component} of a $d$-complex $X$ in $\cK_n$ is an inclusionwise maximal strongly connected subcomplex of $X$.
Given a $d$-complex $X$ in $\cK_n$ and $\sg\in F_{d-1}(X)$, let $C(\sg,X)$ denote the strongly connected component of $X$ containing $\sg$.
The following example relates to the notions of $M$-bounded persistence diagram and $M$-bounded persistent Betti number, recently introduced in~\cite{BCHS20}.
Below, $Z^{d-1}(\cdot)$ indicates the cocycle group of a finite simplicial complex with coefficients in $\R$. Notice that this is actually a vector space since we are using real-valued coefficients and we denote its dimension by $\dim Z^{d-1}(\cdot)$.
\begin{eg}(Local Betti number)
For a fixed $M\in\N$, we define
\[
f(Y_n)=\sum_{\sg\in F_{d-1}(\cK_n)}g((Y_n,\sg)_M),
\]
where we set
\[
g(X)=\frac{\dim Z^{d-1}(X)}{f_{d-1}(X)}\1_{\{0<f_{d-1}(X)\le M\}}
\]
for any finite simplicial complex $X$.
Recall that $f_{d-1}(X)$ indicates the number of $(d-1)$-simplices in $X$.
Clearly, $g$ is invariant under simplicial isomorphisms.
Noting that both the events
\[
\{f_{d-1}((Y_n,\sg)_M)\le M\}
\quad\text{and}\quad
\{f_{d-1}(C(\sg,Y_n))\le M\}
\]
imply the event $\{(Y_n,\sg)_M=C(\sg,Y_n)\}$, we have
\begin{align}\label{eq:local_cocycle}
f(Y_n)
&=\sum_{\sg\in F_{d-1}(\cK_n)}\frac{\dim Z^{d-1}((Y_n,\sg)_M)}{f_{d-1}((Y_n,\sg)_M)}\1_{\{f_{d-1}((Y_n,\sg)_M)\le M\}}\nonumber\\
&=\sum_{\sg\in F_{d-1}(\cK_n)}\frac{\dim Z^{d-1}(C(\sg,Y_n))}{f_{d-1}(C(\sg,Y_n))}\1_{\{f_{d-1}(C(\sg,Y_n))\le M\}}\nonumber\\
&=\sum_C\dim Z^{d-1}(C),
\end{align}
where $C$ ranges over all strongly connected components of $Y_n$ with $f_{d-1}(C)\le M$.
We call the statistic $f(Y_n)$ the \textit{$(d-1)$st $M$-bounded cocycle count} of $Y_n$.
When $d=1$, $f(Y_n)$ is the number of connected components for the Erd\H os--R\'enyi graph model $G(n,\lm/n)$ of size at most $M$.
We also define
\[
\b_{d-1}^M(Y_n)
:=f(Y_n)-\dim B^{d-1}(Y_n)
=f(Y_n)-\binom{n-1}{d-1},
\]
and call it the \textit{$(d-1)$st $M$-bounded Betti number} of $Y_n$.
When $M\ge\binom nd$, the $(d-1)$st $M$-bounded Betti number of $Y_n$ is identical to the usual $(d-1)$st Betti number of $Y_n$ since $f(Y_n)=\dim Z^{d-1}(Y_n)$ from~\eqref{eq:local_cocycle} in this case.
If we delete a $d$-simplex from a $d$-complex $Y$ in $\cK_n$, then the number of strongly connected components $C$ of $Y$ with $f_{d-1}(C)\le M$ increases by at most $d+1$, and $\dim Z^{d-1}(C)\le f_{d-1}(C)\le M$ holds for such $C$.
Hence, $f$ is Lipschitz with $H\equiv(d+1)M$, and $J_n=(d+1)^6M^6$.

Next, we claim that $\var(f(Y_n))=\Om(n^d)$.
To this end, we will use~\eqref{eq:var_lb_unweight}:
\[
\var(f(Y_n))\geq2\binom n{d+1}\frac\lm n\biggl(1-\frac\lm n\biggr)\E[D_{\tau}f(Y_n)]^2,
\]
where $\tau$ is an arbitrarily fixed $d$-simplex in $\cK_n$.
Let $C_1,C_2,\ldots,C_m$ be the strongly connected components of $Y_n-\tau$ that include a $(d-1)$-dimensional face of $\tau$.
If $f_{d-1}(\{\tau\}\cup C_1\cup\cdots\cup C_m)>M$, then obviously $D_\tau f(Y_n)\le0$ by~\eqref{eq:local_cocycle}.
If $f_{d-1}(\{\tau\}\cup C_1\cup\cdots C_m)\le M$, then
\[
D_\tau f(Y_n)=\dim Z^{d-1}(\{\tau\}\cup C_1\cup\cdots\cup C_m)-\dim Z^{d-1}(C_1\cup C_2\cup\cdots\cup C_m)\le 0,
\]
as adding $d$-simplex only reduces the dimension of the cocycle group. Therefore, it holds that $D_\tau f(Y_n)\le0$.
Moreover, if we set $A$ as the event that all $(d-1)$-dimensional faces of $\tau$ are maximal, then $D_{\tau}f(Y_n)=-1$ on $A$.
Therefore,
\[
\E[D_{\tau}f(Y_n)]
\le-\P(A)
=-\biggl(1-\frac\lm n\biggr)^{(n-d-1)(d+1)}.
\]
From the above estimates, we obtain
\[
\frac{\var(f(Y_n))}{n^d}
\ge\frac2{n^d}\binom n{d+1}\frac\lm n\biggl(1-\frac\lm n\biggr)\biggl(1-\frac\lm n\biggr)^{2(n-d-1)(d+1)}
\xrightarrow[n\to\infty]{}\frac{2\lm}{(d+1)!}e^{-2(d+1)\lm}>0.
\]
Consequently, we have $\var(f(Y_n))=\Om(n^d)$.
By Theorem~\ref{thm:CLT_M-bounded}, we obtain
\[
d_K\Biggl(\cL\Biggl(\frac{\b_{d-1}^M(Y_n)-\E[\b_{d-1}^M(Y_n)]}{\sqrt{\var(\b_{d-1}^M(Y_n))}}\Biggr),\cN(0,1)\Biggr)
=d_K\Biggl(\cL\Biggl(\frac{f(Y_n)-\E[f(Y_n)]}{\sqrt{\var(f(Y_n))}}\Biggr),\cN(0,1)\Biggr)
=O(n^{-1/12}).
\]
Furthermore, we remark that $M$ in the above argument may depend on $n$. For example, if $\lm\le1/d$, then the above CLT holds for $M = M_n = o(n^{1/17})$. If $\lm>1/d$, then it is straightforward to verify that if $M=M_n\le c\log n$ for some $c<1/\{4\log( d\lm)\}$, then
\[
\lim_{n\to\infty}d_K\Biggl(\cL\Biggl(\frac{\b_{d-1}^M(Y_n)-\E[\b_{d-1}^M(Y_n)]}{\sqrt{\var(\b_{d-1}^M(Y_n))}}\Biggr),\cN(0,1)\Biggr)
=\lim_{n\to\infty}d_K\Biggl(\cL\Biggl(\frac{f(Y_n)-\E[f(Y_n)]}{\sqrt{\var(f(Y_n))}}\Biggr),\cN(0,1)\Biggr)
=0.
\]
\end{eg}
\section{Proofs}
\label{sec:proofs}
In each of the following subsections, we prove Theorem~\ref{thm:rdm_deriv}, Theorem~\ref{thm:add-one}, Corollaries~\ref{cor:no_gm_rho} and~\ref{cor:suff_CLT}, and Lemma~\ref{lem:disjoint}, respectively. The last lemma is a technical lemma needed in the proof of Theorem~\ref{thm:rdm_deriv}.

\subsection{Proof of Theorem~\ref{thm:rdm_deriv}}
In this subsection, we prove Theorem~\ref{thm:rdm_deriv}.
The proof is based on the idea used in~\cite{Ca21}.
The starting point of the proof is the following lemma, which is obtained by applying Corollary~3.2 of~\cite{Ch14} to our setting. 
\begin{lem}\label{lem:Chatterjee}
For each $\tau,\tau'\in F_d(\cK_n)$, let $c(\tau,\tau')$ be a constant satisfying that for any $F\subset F_d(\cK_n)\setminus\{\tau\}$ and $F'\subset F_d(\cK_n)\setminus\{\tau'\}$,
\[
\sg_n^{-4}\cov(\Dl_\tau f(\X_n)\Dl_\tau f(\X_n^F),\Dl_{\tau'} f(\X_n)\Dl_{\tau'} f(\X_n^{F'}))\le c(\tau,\tau').
\]
Then,
\begin{align}\label{eq:start}
&d_K\biggl(\cL\biggl(\frac{f(\X_n)-\E[f(\X_n)]}{\sg_n}\biggr),\cN(0,1)\biggr)\nonumber\\
&\le\sqrt2 \Biggl(\sum_{\tau,\tau'\in F_d(\cK_n)}c(\tau,\tau')\Biggr)^{1/4}+\Biggr(\frac1{\sg_n^3}\sum_{\tau\in F_d(\cK_n)}\E[|\Dl_\tau f(\X_n)|^3]\Biggr)^{1/2}.
\end{align}
\end{lem}

As we will see later, our main task will be to find $c(\tau,\tau')$ for disjoint $d$-simplices $\tau,\tau'\in F_d(\cK_n)$.
Before proceeding to the main task, we estimate the second term in the right-hand side of~\eqref{eq:start} using the Lipschitzness of $f$.
For convenience, we define events
\begin{equation}
\label{e:antau}
A_{n,\tau}\coloneqq\{b_{n,\tau}+b'_{n,\tau}=1\}\quad\text{and}\quad\tilde A_{n,\tau}\coloneqq\{b_{n,\tau}\vee b'_{n,\tau}=1\}
\end{equation}
for any $\tau\in F_d(\cK_n)$.
Clearly, $\P(A_{n,\tau})\le\P(\tilde A_{n,\tau})\le2p_n$ holds.
\begin{lem}
It holds that
\[
\sum_{\tau\in F_d(\cK_n)}\E[|\Dl_\tau f(\X_n)|^3]\le J_n^{1/2}n^{d+1}p_n.
\]
\end{lem}
\begin{proof}
From the Lipschitzness of $f$ and the mutual independence of $B_n$, $W_n$, $B'_n$, and $W'_n$,
\begin{align*}
\sum_{\tau\in F_d(\cK_n)}\E[|\Dl_\tau f(\X_n)|^3]
&\le\sum_{\tau\in F_d(\cK_n)}\E[\1_{\tilde A_{n,\tau}}H(w_{n,\tau},w'_{n,\tau})^3]\\
&\le\sum_{\tau\in F_d(\cK_n)}2p_n\E[H(w_{n,\tau},w'_{n,\tau})^3]\\
&\le\binom{n}{d+1}2p_nJ_n^{1/2}.
\end{align*}
In the last line, we use the Cauchy--Schwarz inequality.
\end{proof}

Next, we consider $c(\tau,\tau')$ for $d$-simplices $\tau,\tau'\in F_d(\cK_n)$ with $\tau\cap\tau'\neq\emptyset$.
\begin{lem}
Let $\tau,\tau'\in F_d(\cK_n)$ with $\tau\cap\tau'\neq\emptyset$.
Then, we may take
\[
c(\tau,\tau')=\begin{cases}
8J_n^{2/3}p_n^2/\sg_n^4     &\text{if }\tau\neq\tau',\\
2J_n^{2/3}p_n/\sg_n^4       &\text{if }\tau=\tau'.
\end{cases}
\]
In particular,
\[
\sum_{\substack{\tau,\tau'\in F_d(\cK_n)\\\tau\cap\tau'\neq\emptyset}}c(\tau,\tau') \le\frac C{\sg_n^4}J_n^{2/3}(n^{2d+1}p_n^2+n^{d+1}p_n)
=C\biggl(\frac{n^d}{\sg_n^2}\biggr)^2J_n^{2/3}\biggl(\frac{\lm_n^2}n+\frac{\lm_n}{n^d}\biggr).
\]
\end{lem}
\begin{proof}
We first suppose that $\tau\neq\tau'$.
Let $F\subset F_d(\cK_n)\setminus\{\tau\}$ and $F'\subset F_d(\cK_n)\setminus\{\tau'\}$.
From the Lipschitzness of $f$ and the mutual independence of $B_n$, $W_n$, $B'_n$, and $W'_n$,, we have
\begin{align*}
\E[|\Dl_\tau f(\X_n)\Dl_\tau f(\X_n^F)\Dl_{\tau'} f(\X_n)\Dl_{\tau'} f(\X_n^{F'})|]
&\le\E[\1_{\tilde A_{n,\tau}}H(w_{n,\tau},w'_{n,\tau})^2\1_{\tilde A_{n,\tau'}}H(w_{n,\tau'},w'_{n,\tau'})^2]\\
&\le4p_n^2\E[H(w_{n,\tau},w'_{n,\tau})^2]\E[H(w_{n,\tau'},w'_{n,\tau'})^2] \\
&\le4J_n^{2/3}p_n^2,
\end{align*}
where in the last line, we use H\"older's inequality. In the same manner, we have
\[
\E[|\Dl_\tau f(\X_n)\Dl_\tau f(\X_n^F)|]\E[|\Dl_{\tau'} f(\X_n)\Dl_{\tau'} f(\X_n^{F'})|]\le4J_n^{2/3}p_n^2.
\]
Thus, $\cov(\Dl_\tau f(\X_n)\Dl_\tau f(\X_n^F),\Dl_{\tau'} f(\X_n)\Dl_{\tau'} f(\X_n^{F'}))\le8J_n^{2/3}p_n^2$.

Next, we suppose that $\tau=\tau'$, and let $F$, $F'\subset F_d(\cK_n)\setminus\{\tau\}$.
Then, we have
\begin{align*}
&\cov(\Dl_\tau f(\X_n)\Dl_\tau f(\X_n^F),\Dl_\tau f(\X_n)\Dl_\tau f(\X_n^{F'}))\\
&\le\var(\Dl_\tau f(\X_n)\Dl_\tau f(\X_n^F))^{1/2}\var(\Dl_\tau f(\X_n)\Dl_\tau f(\X_n^{F'}))^{1/2}\\
&\le\E[(\Dl_\tau f(\X_n)\Dl_\tau f(\X_n^F))^2]^{1/2}\E[(\Dl_\tau f(\X_n)\Dl_\tau f(\X_n^{F'}))^2]^{1/2}\\
&\le\E[\1_{\tilde A_{n,\tau}}H(w_{n,\tau},w'_{n,\tau})^4]\\
&\le2J_n^{2/3}p_n,
\end{align*}
which completes the proof.
\end{proof}

Finally, we consider $c(\tau,\tau')$ for disjoint $d$-simplices $\tau,\tau'$ in $\cK_n$.
\begin{lem}\label{lem:disjoint}
Let $\tau,\tau'$ be disjoint $d$-simplices in $\cK_n$.
Then, we may take
\[
c(\tau,\tau')=\frac C{\sg_n^4}[\bigl(J_n^{1/2}\dl_{k,n}^{1/2}+\rho_{k,n}+J_n^{2/3}\gm_{k,n}^{1/2} \bigr) p_n^2 +J_n^{2/3}p_n^3].
\]
In particular,
\[
\sum_{\substack{\tau,\tau'\in F_d(\cK_n)\\\tau\cap\tau'=\emptyset}}c(\tau,\tau') \le C\biggl(\frac{n^d}{\sg_n^2}\biggr)^2\biggl[\bigl(J_n^{1/2}\dl_{k,n}^{1/2}+\rho_{k,n}+J_n^{2/3}\gm_{k,n}^{1/2} \bigr) \lm_n^2 +\frac{J_n^{2/3}\lm_n^3}n\biggr].
\]
\end{lem}
We defer the proof of this lemma to Section \ref{sec:prooflemmadisjoint}.
Combining Lemmas~\ref{lem:Chatterjee}--\ref{lem:disjoint}, we can immediately obtain Theorem~\ref{thm:rdm_deriv}.
\begin{proof}[Proof of Theorem~\ref{thm:rdm_deriv}]
Combining Lemmas~\ref{lem:Chatterjee}--\ref{lem:disjoint},
\begin{align*}
&d_K\biggl(\cL\biggl(\frac{f(\X_n)-\E[f(\X_n)]}{\sg_n}\biggr),\cN(0,1)\biggr)\\
&\le\sqrt2\Biggl(\sum_{\substack{\tau,\tau'\in F_d(\cK_n)\\\tau\cap\tau'\neq\emptyset}}c(\tau,\tau')+\sum_{\substack{\tau,\tau'\in F_d(\cK_n)\\\tau\cap\tau'=\emptyset}}c(\tau,\tau')\Biggr)^{1/4}+\Biggr(\frac1{\sg_n^3}\sum_{\tau\in F_d(\cK_n)}\E[|\Dl_\tau f(\X_n)|^3]\Biggr)^{1/2}\\
&\le C\biggl(\frac{n^d}{\sg_n^2}\biggr)^{1/2}\biggl[J_n^{2/3}\biggl(\frac{\lm_n^2}n+\frac{\lm_n}{n^d}\biggr)+\bigl(J_n^{1/2}\dl_{k,n}^{1/2}+\rho_{k,n}+J_n^{2/3}\gm_{k,n}^{1/2}\bigr)\lm_n^2 +\frac{J_n^{2/3}\lm_n^3}n\biggr]^{1/4}
+(\sg_n^{-3}J_n^{1/2}n^{d+1}p_n)^{1/2}\\
&=C\biggl(\frac{n^d}{\sg_n^2}\biggr)^{1/2}\biggl[\bigl(J_n^{1/2}\dl_{k,n}^{1/2}+\rho_{k,n}+J_n^{2/3}\gm_{k,n}^{1/2}\bigr)\lm_n^2+J_n^{2/3}\biggl(\frac{\lm_n^2}n+\frac{\lm_n}{n^d}+\frac{\lm_n^3}n\biggr)\biggr]^{1/4}
+\biggl(\frac{n^d}{\sg_n^2}\biggr)^{3/4}\frac{J_n^{1/4}\lm_n^{1/2}}{n^{d/4}},
\end{align*}
which completes the proof.
\end{proof}

\subsection{Proof of Theorem~\ref{thm:add-one}}\label{ssec:add-one}
In this subsection, we prove Theorem~\ref{thm:add-one} using Theorem~\ref{thm:rdm_deriv}.
Recall that for each $k\in\N$,
\begin{align*}
\tilde\dl_{k,n}&\coloneqq\max_{i=0,1}\E[\{D_\tau f(\X_n)-D_\tau f(B_k(\tau,\X_n))\}^2\mid b_{n,\tau'}=i]
\shortintertext{and}
\tilde\rho_{k,n}&\coloneqq\sup_{F,F'\subset F_d(\cK_n)\setminus\{\tau,\tau'\}}\cov(D_\tau f(B_k(\tau,\X_n))D_\tau f(B_k(\tau,\X_n^F)),D_{\tau'} f(B_k(\tau',\X_n))D_{\tau'} f(B_k(\tau',\X_n^{F'}))),
\end{align*}
where $\tau,\tau'\in F_d(\cK_n)$ are disjoint $d$-simplices. For $\tau\in F_d(\cK_n)$ and $F\subset F_d(\cK_n)\setminus\{\tau\}$, we define
\[
L_k^\tau(\X_n^F)=\Dl_\tau f(B_k(\tau,\X_n^F))
\quad\text{and}\quad
\hat L_k^\tau(\X_n^F)=D_\tau f(B_k(\tau,\X_n^F)).
\]
Here, $L$ stands for ``local''.
In the following, we also write
\begin{equation}
\label{e:boldbntau}    
\bfb_{n,\tau}=(b_{n,\tau},b'_{n,\tau})
\end{equation}
for $\tau\in F_d(\cK_n)$ for simplicity. Recall definitions of $A_{n,\tau}$ and $\tilde A_{n,\tau}$ from \eqref{e:antau}.
\begin{proof}[Proof of Theorem~\ref{thm:add-one}]
Let $\tau,\tau'\in F_d(\cK_n)$ be disjoint $d$-simplices, and let $F\subset F_d(\cK_n)\setminus\{\tau\}$ and $F'\subset F_d(\cK_n)\setminus\{\tau'\}$ be fixed. Below, we write $\overset{*}{\sum}_{i_1,i_2,i_3}$ for summations taken over all $(i_1,i_2,i_3) \in\{0,1\}^3$ such that $i_1+i_2=1$.
Since
\begin{align*}
&\1_{A_{n,\tau}}\E[\{\Dl_\tau f(\X_n^F)-\Dl_\tau f(B_k(\tau,\X_n))\}^2\mid\bfb_{n,\tau},b_{n,\tau'}]\\
&=\sum^*_{i_1,i_2,i_3}\E[\{\Dl_\tau f(\X_n)-\Dl_\tau f(B_k(\tau,\X_n))\}^2\mid\bfb_{n,\tau}=(i_1,i_2),b_{n,\tau'}=i_3]\1_{\{(\bfb_{n,\tau}=(i_1,i_2), b_{n,\tau'}=i_3\}}\\
&=\sum^*_{i_1,i_2,i_3}\E[\{D_\tau f(\X_n)-D_\tau f(B_k(\tau,\X_n))\}^2\mid\bfb_{n,\tau}=(i_1,i_2),b_{n,\tau'}=i_3]\1_{\{(\bfb_{n,\tau}=(i_1,i_2), b_{n,\tau'}=i_3\}}\\
&=\sum^*_{i_1,i_2,i_3}\E[\{D_\tau f(\X_n)-D_\tau f(B_k(\tau,\X_n))\}^2\mid b_{n,\tau'}=i_3]\1_{\{(\bfb_{n,\tau}=(i_1,i_2),b_{n,\tau'}=i_3\}}\\
&\le\1_{A_{n,\tau}}\tilde\dl_{k,n},
\end{align*}
we may take $\tilde\dl_{k,n}$ as the constant $\dl_{k,n}$ in Theorem~\ref{thm:rdm_deriv}.

We next show that on the event $A_{n,\tau}\cap A_{n,\tau'}$,
\begin{equation}\label{eq:rho_est}
\cov(L_k^\tau(\X_n)L_k^\tau(\X_n^F),L_k^{\tau'}(\X_n)L_k^{\tau'}(\X_n^{F'})\mid\bfb_{n,\tau},\bfb_{n,\tau'})
\le\tilde\rho_{k,n}+CJ_n^{2/3}\gm_{k,n}^{1/3}.
\end{equation}
Note first that
\begin{align*}
|L_k^\tau(\X_n)L_k^\tau(\X_n^F)|,|L_k^\tau(\X_n-\tau')L_k^\tau(\X_n^F-\tau')|&\le\1_{\tilde A_{n,\tau}} H(w_{n,\tau},w'_{n,\tau})^2
\shortintertext{and}
|L_k^{\tau'}(\X_n)L_k^{\tau'}(\X_n^{F'})|,|L_k^{\tau'}(\X_n-\tau)L_k^{\tau'}(\X_n^{F'}-\tau)|&\le\1_{\tilde A_{n,\tau'}} H(w_{n,\tau'},w'_{n,\tau'})^2
\end{align*}
from the Lipschitzness of $f$.
Furthermore, since
\begin{align*}
&\{L_k^\tau(\X_n)L_k^\tau(\X_n^F)\neq L_k^\tau(\X_n-\tau')L_k^\tau(\X_n^F-\tau')\}\\
&=\{\Dl_\tau f(B_k(\tau,\X_n))\Dl_\tau f(B_k(\tau,\X_n^F))\neq \Dl_\tau f(B_k(\tau,\X_n-\tau'))\Dl_\tau f(B_k(\tau,\X_n^F-\tau'))\}\\
&\subset\{\tau'\in B_k(\tau,\X_n)\}\cup\{\tau'\in B_k(\tau,\X_n^F)\}\\
&\subset\bigcup_{\substack{\sg\colon\text{$(d-1)$-face of }\tau\\\sg'\colon\text{$(d-1)$-face of }\tau'}}\{\text{$\sg\xleftrightarrow[]{\X_n}\sg'$ or $\sg\xleftrightarrow[]{\X_n^F}\sg'$ by a path of length at most $k$}\},
\end{align*}
we have
\begin{align}\label{eq:prob_LL-LL}
&\P(L_k^\tau(\X_n)L_k^\tau(\X_n^F)\neq L_k^\tau(\X_n-\tau')L_k^\tau(\X_n^F-\tau')\mid\bfb_{n,\tau},\bfb_{n,\tau'})\nonumber\\
&=\1_{\tilde A_{n,\tau}}\P(L_k^\tau(\X_n)L_k^\tau(\X_n^F)\neq L_k^\tau(\X_n-\tau')L_k^\tau(\X_n^F-\tau')\mid\bfb_{n,\tau},\bfb_{n,\tau'})\nonumber\\
&\le\1_{\tilde A_{n,\tau}}\sum_{\substack{\sg\colon\text{$(d-1)$-face of }\tau\\\sg'\colon\text{$(d-1)$-face of }\tau'}}\P(\text{$\sg\xleftrightarrow[]{\X_n}\sg'$ or $\sg\xleftrightarrow[]{\X_n^F}\sg'$ by a path of length at most $k$})\nonumber\\
&\le2\1_{\tilde A_{n,\tau}}\sum_{\substack{\sg\colon\text{$(d-1)$-face of }\tau\\\sg'\colon\text{$(d-1)$-face of }\tau'}}\P(\text{$\sg\xleftrightarrow[]{\X_n}\sg'$ by a path of length at most $k$})\nonumber\\
&=2(d+1)^2\gm_{k,n}\1_{\tilde A_{n,\tau}}.
\end{align}
Similarly, we have
\[
\P(L_k^{\tau'}(\X_n)L_k^{\tau'}(\X_n^{F'})\neq L_k^{\tau'}(\X_n-\tau)L_k^{\tau'}(\X_n^{F'}-\tau)\mid\bfb_{n,\tau},\bfb_{n,\tau'})
\le2(d+1)^2\gm_{k,n}\1_{\tilde A_{n,\tau'}}.
\]
From the above estimates, a straightforward calculation involving the conditional H\"older inequality yields
\begin{align}\label{eq:cdcov_modi}
&\cov(L_k^\tau(\X_n)L_k^\tau(\X_n^F),L_k^{\tau'}(\X_n)L_k^{\tau'}(\X_n^{F'})\mid\bfb_{n,\tau},\bfb_{n,\tau'})\nonumber\\
&\le\cov(L_k^\tau(\X_n-\tau')L_k^\tau(\X_n^F-\tau'),L_k^{\tau'}(\X_n-\tau)L_k^{\tau'}(\X_n^{F'}-\tau)\mid\bfb_{n,\tau},\bfb_{n,\tau'})+CJ_n^{2/3}\gm_{k,n}^{1/3}\1_{\tilde A_{n,\tau}}\1_{\tilde A_{n,\tau'}}.
\end{align}
We now use $\overset{*}{\sum}_{i_1,i_2,i_3,i_4}$ to denote summations taken over all $i_1,i_2,i_3,i_4$ such that $i_1+i_2=i_3+i_4=1$.
Furthermore,
\begin{align}\label{eq:cdcov_cov}
&\1_{A_{n,\tau}}\1_{A_{n,\tau'}}\cov(L_k^\tau(\X_n-\tau')L_k^\tau(\X_n^F-\tau'),L_k^{\tau'}(\X_n-\tau)L_k^{\tau'}(\X_n^{F'}-\tau)\mid\bfb_{n,\tau},\bfb_{n,\tau'})\nonumber\\
&=\sum^*_{i_1,i_2,i_3,i_4}\cov\left(L_k^\tau(\X_n-\tau')L_k^\tau(\X_n^F-\tau'),L_k^{\tau'}(\X_n-\tau)L_k^{\tau'}(\X_n^{F'}-\tau)\relmiddle|\substack{\bfb_{n,\tau}=(i_1,i_2)\\\bfb_{n,\tau'}=(i_3,i_4)}\right)\1_{\Bigl\{\substack{\bfb_{n,\tau}=(i_1,i_2)\\\bfb_{n,\tau'}=(i_3,i_4)}\Bigr\}}\nonumber\\
&=\sum^*_{i_1,i_2,i_3,i_4}\cov\left(\hat L_k^\tau(\X_n-\tau')\hat L_k^\tau(\X_n^F-\tau'),\hat L_k^{\tau'}(\X_n-\tau)\hat L_k^{\tau'}(\X_n^{F'}-\tau)\relmiddle|\substack{\bfb_{n,\tau}=(i_1,i_2)\\\bfb_{n,\tau'}=(i_3,i_4)}\right)\1_{\Bigl\{\substack{\bfb_{n,\tau}=(i_1,i_2)\\\bfb_{n,\tau'}=(i_3,i_4)}\Bigr\}}\nonumber\\
&=\sum^*_{i_1,i_2,i_3,i_4}\cov(\hat L_k^\tau(\X_n-\tau')\hat L_k^\tau(\X_n^F-\tau'),\hat L_k^{\tau'}(\X_n-\tau)\hat L_k^{\tau'}(\X_n^{F'}-\tau))\1_{\Bigl\{\substack{\bfb_{n,\tau}=(i_1,i_2)\\\bfb_{n,\tau'}=(i_3,i_4)}\Bigr\}}\nonumber\\
&=\1_{A_{n,\tau}}\1_{A_{n,\tau'}}\cov(\hat L_k^\tau(\X_n-\tau')\hat L_k^\tau(\X_n^F-\tau'),\hat L_k^{\tau'}(\X_n-\tau)\hat L_k^{\tau'}(\X_n^{F'}-\tau))\nonumber\\
&=\1_{A_{n,\tau}}\1_{A_{n,\tau'}}\cov(\hat L_k^\tau(\X_n-\tau')\hat L_k^\tau(\X_n^{F\setminus\{\tau'\}}-\tau'),\hat L_k^{\tau'}(\X_n-\tau)\hat L_k^{\tau'}(\X_n^{F'\setminus\{\tau\}}-\tau)).
\end{align}
Similar to~\eqref{eq:cdcov_modi}, a simple calculation involving H\"older's inequality yields
\begin{align}\label{eq:cov_modi}
&\cov(\hat L_k^\tau(\X_n-\tau')\hat L_k^\tau(\X_n^{F\setminus\{\tau'\}}-\tau'),\hat L_k^{\tau'}(\X_n-\tau)\hat L_k^{\tau'}(\X_n^{F'\setminus\{\tau\}}-\tau))\nonumber\\
&\le\cov(\hat L_k^\tau(\X_n)\hat L_k^\tau(\X_n^{F\setminus\{\tau'\}}),\hat L_k^{\tau'}(\X_n)\hat L_k^{\tau'}(\X_n^{F'\setminus\{\tau\}}))
+CJ_n^{2/3}\gm_{k,n}^{1/3}.
\end{align}
Combining~\eqref{eq:cdcov_modi},~\eqref{eq:cdcov_cov}, and~\eqref{eq:cov_modi} with the definition of $\tilde\rho_{k,n}$, we obtain~\eqref{eq:rho_est}.
Consequently, we may take $\tilde\dl_{k,n}$ and $\tilde\rho_{k,n}+CJ_n^{2/3}\gm_{k,n}^{1/3}$ as the constant $\dl_{k,n}$ and $\rho_{k,n}$ in Theorem~\ref{thm:rdm_deriv}, respectively, which completes the proof as mentioned in Remark~\ref{rem:add-one}.
\end{proof}

\subsection{Proof of Corollaries~\ref{cor:no_gm_rho} and~\ref{cor:suff_CLT}}
In this subsection, we prove Corollaries~\ref{cor:no_gm_rho} and~\ref{cor:suff_CLT}.
We begin by proving Corollary~\ref{cor:no_gm_rho} by estimating $\gm_{k,n}$ and $\tilde\rho_{k,n}$.
\begin{prop}\label{prop:gm_est}
For $k\in\N$, it holds that
\[
\gm_{k,n}
\le\frac{k^{d+1}(1\vee d\lm_n)^k}{n^d}.
\]
In particular, if $k\le n$, then
\[
\gm_{k,n}
\le\frac{k^2(1\vee d\lm_n)^k}n.
\]
\end{prop}
\begin{proof}
Let $\sg$ and $\sg'$ be disjoint $(d-1)$-simplices in $\cK_n$.
Recall that $\gm_{k,n}$ is the probability that $\sg$ and $\sg'$ are connected by a path of length at most $k$. Therefore,
\begin{align}\label{eq:gm_est}
\gm_{k,n}
&\le\sum_{i=d}^k\P\bigl(\text{$\sg\xleftrightarrow[]{\X_n}\sg'$ by a path of length $i$}\bigr)\nonumber\\
&\le \sum_{i=d}^k\binom id(dn)^{i-d}d^dp_n^i
= \frac1{n^d}\sum_{i=d}^k\binom id(d\lm_n)^i
\le \frac{k^d}{n^d}\sum_{i=d}^k(d\lm_n)^i
= \frac{k^{d+1}(1\vee d\lm_n)^k}{n^d}.
\end{align}
For the second inequality, we note that there are at most $nd$ ways of choosing the next $(d-1)$-simplex in the generating process of a path and $\binom id d^d$ is an upper bound on the number of the ways that at some $d$ steps in the generating process the vertices of $\sigma'$ are chosen so that the last $(d-1)$-simplex of the path coincides with $\sg'$.
\end{proof}

Next, we will estimate
\[
\tilde\rho_{k,n}\coloneqq\sup_{F,F'\subset F_d(\cK_n)\setminus\{\tau,\tau'\}}\cov(D_\tau f(B_k(\tau,\X_n))D_\tau f(B_k(\tau,\X_n^F)),D_{\tau'} f(B_k(\tau',\X_n))D_{\tau'} f(B_k(\tau',\X_n^{F'}))),
\]
where $\tau,\tau'$ are disjoint $d$-simplices in $\cK_n$ and again the covariance is independent of the chosen $d$-simplices.
\begin{prop}\label{prop:tilrho_est}
For $k \in \N$ with $k \leq n$, it holds that
\[
\tilde\rho_{k,n}
\le CJ_n^{2/3}\biggl(\frac{k^5(1\vee  d\lm_n)^{2k}}n\biggr)^{1/3}.
\]
\end{prop}
For the proof, the following two lemmas are crucial.
Let $\tau,\tau'$ be disjoint $d$-simplices in $\cK_n$, and let $F,F'\subset F_d(\cK_n)\setminus\{\tau,\tau'\}$ be fixed.
For $k\in\Z_{\ge0}$, define
\[
\bfB_k^{\tau,F}\coloneqq ((\X_n,\tau)_k,(\X_n^F,\tau)_k)
\quad\text{and}\quad
\bfB_k^{\tau',F'}\coloneqq ((\X_n,\tau')_k,(\X_n^{F'},\tau')_k).
\]
Recall the definition of $(\X_n,\tau)_k$ in Definition~\ref{df:M-ball}.  
By forgetting the weights of $\bfB_k^{\tau,F}$ and $\bfB_k^{\tau',F'}$, we also define $B_k^{\tau,F}$ and $B_k^{\tau',F'}$, respectively. In addition, let $I_k$ denote the event that two (weighted) simplicial complexes $(\X_n,\tau)_k\cup(\X_n^F,\tau)_k$ and $(\X_n,\tau')_k\cup(\X_n^{F'},\tau')_k$ have no common vertices. 
\begin{lem}\label{lem:nonintersect}
For $k \in \N$ with $k \leq n$, it holds that
\[
\P(I_k^c)\le\frac{Ck^2(1\vee  d\lm_n)^{2k}}n.
\]
\end{lem}
\begin{proof}
If the two simplicial complexes $(\X_n,\tau)_k\cup(\X_n^F,\tau)_k$ and $(\X_n,\tau')_k\cup(\X_n^{F'},\tau')_k$ have a common vertex, then at least one of the following four simplicial complexes is not the empty complex (i.e., has nonempty vertex set):
\begin{align*}
(\X_n,\tau)_k\cap(\X_n,\tau')_k
\text{, }(\X_n,\tau)_k\cap(\X_n^{F'},\tau')_k
\text{, }(\X_n^F,\tau)_k\cap(\X_n,\tau')_k
\text{, and }(\X_n^F,\tau)_k\cap(\X_n^{F'},\tau')_k.
\end{align*}
Hence, it suffices to estimate only the probability that $(\X_n^F,\tau)_k$ and $(\X_n^{F'},\tau')_k$ have a common vertex.
Now, we note that if $(\X_n^F,\tau)_k$ and $(\X_n^{F'},\tau')_k$ have a common vertex $v\in[n]\setminus(\tau\cup\tau')$, then there exist paths $w$ in $\X_n^F$ and $w'$ in $\X_n^{F'}$ of length at most $k$ such that $w$ and $w'$ possess $v$ in common but not any $d$-simplices, and contain some faces of $\tau$ and $\tau'$, respectively.
In addition, if $(\X_n^F,\tau)_k$ and $(\X_n^{F'},\tau')_k$ have a common vertex $v\in\tau'$ (resp. $v\in\tau$), then there exist a path $w$ in $\X_n^F$ (resp. $w'$ in $\X_n^{F'}$) of length at most $k$ such that $w$ (resp. $w'$) possesses $v$ and some face of $\tau$ (resp. $\tau'$).
Therefore, in a similar manner to~\eqref{eq:gm_est}, we can easily verify that
\[
\P((\X_n^F,\tau)_k\cap(\X_n^{F'},\tau')_k\neq\{\emptyset\})
\le\frac{Ck^2(1\vee  d\lm_n)^{2k}}n.\qedhere
\]
\end{proof}

\begin{lem}\label{lem:Cao_Lem6.8}
Let $d\ge2$.
Then, on the event $I_k$, $\bfB_{k+1}^{\tau,F}$ and $\bfB_{k+1}^{\tau',F'}$ are conditionally independent given $B_k^{\tau,F},B_k^{\tau',F'}$, i.e., for any bounded measurable functions $g_\tau$ and $g_{\tau'}$,
\begin{equation}\label{eq:condi_indep}
\1_{I_k}\cov(g_\tau(\bfB_{k+1}^{\tau,F}),g_{\tau'}(\bfB_{k+1}^{\tau',F'})\mid B_k^{\tau,F},B_k^{\tau',F'})=0.
\end{equation}
Furthermore,
\begin{align}
\1_{I_k}\E[g_\tau(\bfB_{k+1}^{\tau,F})\mid B_k^{\tau,F},B_k^{\tau',F'}]
&=\1_{I_k}\E[g_\tau(\bfB_{k+1}^{\tau,F})\mid B_k^{\tau,F}]\label{eq:rdc_condi_1}
\shortintertext{and}
\1_{I_k}\E[g_{\tau'}(\bfB_{k+1}^{\tau',F'})\mid B_k^{\tau,F},B_k^{\tau',F'}]
&=\1_{I_k}\E[g_{\tau'}(\bfB_{k+1}^{\tau',F'})\mid B_k^{\tau',F'}].\label{eq:rdc_condi_2}
\end{align}
\end{lem}
\begin{proof}
The event $I_k$ is a disjoint union (over $(B_1,B_2),(B'_1,B'_2)$) of events $\{B_k^{\tau,F}=(B_1,B_2)\} \cap \{B_k^{\tau',F'}= (B'_1,B'_2)\}$, where $(B_1,B_2)$ and $(B'_1,B'_2)$ are pairs of simplicial complexes such that $B_1,B_2\ni\tau$, $B'_1,B'_2\ni\tau'$, and two simplicial complexes $B_1\cup B_2$ and $B'_1\cup B'_2$ have no common vertices. So it suffices to prove the above claims for $\1_{\{B_k^{\tau,F}= (B_1,B_2)\} \cap \{B_k^{\tau',F'}=(B'_1,B'_2)\}}$ instead of $\1_{I_k}$.

Given a subcomplex $X\subset\cK_n$, define
\[
C(X)\coloneqq\{(b_{n,\sg\cup\{v\}},w_{n,\sg\cup\{v\}},b'_{n,\sg\cup\{v\}},w'_{n,\sg\cup\{v\}})\colon\sg\in F_{d-1}(X),v\in [n]\setminus\sg\}.
\]
Fix two pairs of simplicial complexes $(B_1,B_2)$ and $(B'_1,B'_2)$ such that $B_1,B_2\ni\tau$, $B'_1,B'_2\ni\tau'$, and two simplicial complexes $B_1\cup B_2$ and $B'_1\cup B'_2$ have no common vertices.
Given the events $\{B_k^{\tau,F}= (B_1,B_2)\}$ and $\{B_k^{\tau',F'}= (B'_1,B'_2)\}$, the random variables $g_\tau(\bfB_{k+1}^{\tau,F})$ and $g_{\tau'}(\bfB_{k+1}^{\tau',F'})$ depend only on $C(B_1\cup B_2)$ and $C(B'_1\cup B'_2)$, respectively. Since $d\ge2$, $C(B_1\cup B_2)$ and $C(B'_1\cup B'_2)$ are independent, which with the above observation immediately implies~\eqref{eq:condi_indep}.
Noting also that the events $\{B_k^{\tau,F}=(B_1,B_2)\}$ and $\{B_k^{\tau',F'}=(B'_1,B'_2)\}$, respectively, depends only on $C(B_1\cup B_2)$ and $C(B'_1\cup B'_2)$, we can easily conclude~\eqref{eq:rdc_condi_1} and~\eqref{eq:rdc_condi_2}.
\end{proof}

\begin{proof}[Proof of Proposition~\ref{prop:tilrho_est}]
Let $\tau,\tau'$ be disjoint $d$-simplices in $\cK_n$, and let $F,F'\subset F_d(\cK_n)\setminus\{\tau,\tau'\}$ be fixed.
We write 
\[
g_\tau(\bfB_k^{\tau,F})=D_\tau f(B_k(\tau,\X_n))D_\tau f(B_k(\tau,\X_n^F))
\quad\text{and}\quad
g_{\tau'}(\bfB_k^{\tau',F'})=D_{\tau'} f(B_k(\tau',\X_n))D_{\tau'} f(B_k(\tau',\X_n^{F'})).
\]
From the Lipschitzness of $f$, we have
\[
|g_\tau(\bfB_k^{\tau,F})|\le H(w_{n,\tau},w'_{n,\tau})^2
\quad\text{and}\quad |g_{\tau'}(\bfB_k^{\tau',F'})|\le H(w_{n,\tau'},w'_{n,\tau'})^2.
\]
Therefore, using H\"older's inequality, conditional covariance formula, and Lemma~\ref{lem:Cao_Lem6.8}, we have
\begin{align*}
&\cov(g_\tau(\bfB_k^{\tau,F}),g_{\tau'}(\bfB_k^{\tau',F'}))\\
&\le\cov(\1_{I_{k-1}}g_\tau(\bfB_k^{\tau,F}),\1_{I_{k-1}}g_{\tau'}(\bfB_k^{\tau',F'}))+CJ_n^{2/3}\P(I_{k-1}^c)^{1/3}\\
&=\E[\1_{I_{k-1}}\cov(g_\tau(\bfB_k^{\tau,F}),g_{\tau'}(\bfB_k^{\tau',F'})\mid B_{k-1}^{\tau,F},B_{k-1}^{\tau',F'})]\\
&\qad+\cov(\1_{I_{k-1}}\E[g_\tau(\bfB_k^{\tau,F})\mid B_{k-1}^{\tau,F},B_{k-1}^{\tau',F'}],\1_{I_{k-1}}\E[g_{\tau'}(\bfB_k^{\tau',F'})\mid B_{k-1}^{\tau,F},B_{k-1}^{\tau',F'}])+CJ_n^{2/3}\P(I_{k-1}^c)^{1/3}\\
&=\cov(\1_{I_{k-1}}\E[g_\tau(\bfB_k^{\tau,F})\mid B_{k-1}^{\tau,F}],\1_{I_{k-1}}\E[g_{\tau'}(\bfB_k^{\tau',F'})\mid B_{k-1}^{\tau',F'}])+CJ_n^{2/3}\P(I_{k-1}^c)^{1/3}\\
&\le\cov(\E[g_\tau(\bfB_k^{\tau,F})\mid B_{k-1}^{\tau,F}],\E[g_{\tau'}(\bfB_k^{\tau',F'})\mid B_{k-1}^{\tau',F'}])+CJ_n^{2/3}\P(I_{k-1}^c)^{1/3},
\end{align*}
where in the last line, we also used the estimate
\[
|\E[g_\tau(\bfB_k^{\tau,F})\mid B_{k-1}^{\tau,F}]|
\le\E[H(w_{n,\tau},w'_{n,\tau})^2\mid B_{k-1}^{\tau,F}]
=\E[H(w_{n,\tau},w'_{n,\tau})^2]
\le J_n^{1/3}
\]
and similarly $|\E[g_{\tau'}(\bfB_k^{\tau',F'})\mid B_{k-1}^{\tau',F'}]|\le J_n^{1/3}$.
Note that $\E[g_\tau(\bfB_k^{\tau,F})\mid B_{k-1}^{\tau,F}]$ and $\E[g_{\tau'}(\bfB_k^{\tau',F'})\mid B_{k-1}^{\tau',F'}]$ in the last line are respectively functions of $B_{k-1}^{\tau,F}$ and $B_{k-1}^{\tau',F'}$, especially $\bfB_{k-1}^{\tau,F}$ and $\bfB_{k-1}^{\tau',F'}$.
Therefore, by iterating the above estimate and using Lemma~\ref{lem:nonintersect}, we obtain
\begin{align*}
\cov(g_\tau(\bfB_k^{\tau,F}),g_{\tau'}(\bfB_k^{\tau',F'}))
\le CJ_n^{2/3}\sum_{i=1}^{k-1}\P(I_i^c)^{1/3}
&\le CJ_n^{2/3}\sum_{i=1}^{k-1}\biggl(\frac{i^2(1\vee  d\lm_n)^{2i}}n\biggr)^{1/3}\\
&\le CJ_n^{2/3}\biggl(\frac{k^5(1\vee  d\lm_n)^{2k}}n\biggr)^{1/3}.
\end{align*}
\end{proof}

Now, we are ready to prove Corollaries~\ref{cor:no_gm_rho} and \ref{cor:suff_CLT}.
\begin{proof}[Proof of Corollary~\ref{cor:no_gm_rho}]
Combining Theorem~\ref{thm:add-one} with Propositions~\ref{prop:gm_est} and~\ref{prop:tilrho_est}, we have
\begin{align*}
&d_K\biggl(\cL\biggl(\frac{f(\X_n)-\E[f(\X_n)]}{\sg_n}\biggr),\cN(0,1)\biggr)\nonumber\\
&\le CJ_n^{1/6}\biggl(\frac{n^d}{\sg_n^2}\biggr)^{1/2}\biggl[\biggl\{\tilde\dl_{k,n}^{1/2}+\biggl(\frac{k^5(1\vee d\lm_n)^{2k}}n\biggr)^{1/3}+\biggl(\frac{k^2(1\vee d\lm_n)^k}n\biggr)^{1/3}\biggr\}\lm_n^2+\frac{\lm_n^2}n+\frac{\lm_n}{n^d}+\frac{\lm_n^3}n\biggr]^{1/4}\\
&\qad+\biggl(\frac{n^d}{\sg_n^2}\biggr)^{3/4}\frac{J_n^{1/4}\lm_n^{1/2}}{n^{d/4}}.
\end{align*}
Noting that
\begin{align*}
&\biggl\{\tilde\dl_{k,n}^{1/2}+\biggl(\frac{k^5(1\vee d\lm_n)^{2k}}n\biggr)^{1/3}+\biggl(\frac{k^2(1\vee d\lm_n)^k}n\biggr)^{1/3}\biggr\}\lm_n^2+\frac{\lm_n^2}n+\frac{\lm_n}{n^d}+\frac{\lm_n^3}n\\
&\le\biggl\{\tilde\dl_{k,n}^{1/2}+\biggl(\frac{k^5(1\vee d\lm_n)^{2k}}n\biggr)^{1/3}+\biggl(\frac{k^2(1\vee d\lm_n)^k}n\biggr)^{1/3}+\frac1n+\frac1{n^d}+\frac{1\vee\lm_n}n\biggr\}(1\vee\lm_n)^2\\
&\le C\biggl\{\tilde\dl_{k,n}^{1/2}+\biggl(\frac{k^5(1\vee  d\lm_n)^{2k}}n\biggr)^{1/3}\biggr\}(1\vee\lm_n)^2,
\end{align*}
we immediately obtain~\eqref{eq:no_gm_rho}.
\end{proof}

\begin{proof}[Proof of Corollary~\ref{cor:suff_CLT}]
From Corollary \ref{cor:no_gm_rho}, the conclusion follows immediately in the first case. Under the assumptions of the second or third case, one can easily verify that
\[
\lim_{n\to\infty}\frac{k(n)^5(1\vee d\lm_n)^{2k(n)}}n=0
\]
and this completes the proof using Corollary \ref{cor:no_gm_rho}.
\end{proof}

\subsection{Proof of Lemma~\ref{lem:disjoint}}
\label{sec:prooflemmadisjoint}
In this subsection, we prove Lemma~\ref{lem:disjoint}.
Recall that for $\tau\in F_d(\cK_n)$ and $F\subset F_d(\cK_n)\setminus\{\tau\}$, we define
\[
L_k^\tau(\X_n^F)\coloneqq\Dl_\tau f(B_k(\tau,\X_n^F)).
\]
Additionally, we define
\[
R_k^\tau(\X_n^F)\coloneqq\Dl_\tau f(\X_n^F)-L_k^\tau(\X_n^F).
\]
Here, $L$ and $R$ stand for ``local'' and ``remainder'', respectively. Recall the notations $A_{n,\tau},\tilde A_{n,\tau}$ and $\bfb_{n,\tau}$ from \eqref{e:antau} and \eqref{e:boldbntau}. From the Lipschitzness of $f$, we have
\[
|L_k^\tau(\X_n^F)|\le\1_{\tilde A_{n,\tau}}H(w_{n,\tau},w'_{n,\tau})
\quad\text{and}\quad
|R_k^\tau(\X_n^F)|\le|\Dl_\tau f(\X_n^F)|+|L_k^\tau(\X_n^F)|\le2\1_{\tilde A_{n,\tau}}H(w_{n,\tau},w'_{n,\tau}).
\]
In what follows, let $\tau,\tau'$ be fixed as disjoint $d$-simplices in $\cK_n$, and let $F\subset F_d(\cK_n)\setminus\{\tau\}$ and $F'\subset F_d(\cK_n)\setminus\{\tau'\}$.
We first decompose $\cov(\Dl_\tau f(\X_n)\Dl_\tau f(\X_n^F),\Dl_{\tau'} f(\X_n)\Dl_{\tau'} f(\X_n^{F'}))$ into $16$ terms:
\begin{align}\label{eq:16terms}
\cov(\Dl_\tau f(\X_n)\Dl_\tau f(\X_n^F),\Dl_{\tau'} f(\X_n)\Dl_{\tau'} f(\X_n^{F'}))
&=\cov(L_k^\tau(\X_n)L_k^\tau(\X_n^F),L_k^{\tau'}(\X_n)L_k^{\tau'}(\X_n^{F'}))\nonumber\\
&\qad+\cov(R_k^\tau(\X_n)L_k^\tau(\X_n^F),L_k^{\tau'}(\X_n)L_k^{\tau'}(\X_n^{F'}))\nonumber\\
&\qad+\cov(L_k^\tau(\X_n)R_k^\tau(\X_n^F),L_k^{\tau'}(\X_n)L_k^{\tau'}(\X_n^{F'}))\nonumber\\
&\qad+\cdots\nonumber\\
&\qad+\cov(R_k^\tau(\X_n)R_k^\tau(\X_n^F),R_k^{\tau'}(\X_n)R_k^{\tau'}(\X_n^{F'})).
\end{align}
The first term in the right-hand side of~\eqref{eq:16terms} consists of only local factors, and the other 15 terms include at least one remainder factor.
We will estimate the first term and the other 15 terms in different ways.
In the following, we write $H_n\coloneqq H(w_{n,\tau},w'_{n,\tau})+ H(w_{n,\tau'},w'_{n,\tau'})$ for convenience.
Note that $\E[H_n^6]\le CJ_n$.
We first estimate the latter 15 terms in the right-hand side of~\eqref{eq:16terms}.
\begin{lem}\label{lem:covR}
Each of the last 15 terms in the right-hand side of~\eqref{eq:16terms} is bounded above by
\[
CJ_n^{1/2}p_n^2\dl_{k,n}^{1/2} +CJ_n^{2/3}p_n^3.
\]
\end{lem}
\begin{proof}
We will estimate only the term of the form $\cov(L_k^\tau(\X_n)R_k^\tau(\X_n^F),L_k^{\tau'}(\X_n)L_k^{\tau'}(\X_n^{F'}))$.
The other terms can be estimated in almost the same manner.
From the Lipschitzness of $f$, we have
\begin{align*}
|L_k^\tau(\X_n)R_k^\tau(\X_n^F)L_k^{\tau'}(\X_n)L_k^{\tau'}(\X_n^{F'})|
&\le\1_{\tilde A_{n,\tau}}\1_{\tilde A_{n,\tau'}}H_n^3|R_k^\tau(\X_n^F)|\\
&\le\1_{A_{n,\tau}}\1_{\tilde A_{n,\tau'}}H_n^3|R_k^\tau(\X_n^F)|+2\1_{\{b_{n,\tau}=b'_{n,\tau}=1\}}\1_{\tilde A_{n,\tau'}}H_n^4.
\end{align*}
Therefore,
\begin{align}\label{eq:covR_1}
\E[|L_k^\tau(\X_n)R_k^\tau(\X_n^F)L_k^{\tau'}(\X_n)L_k^{\tau'}(\X_n^{F'})|]
&\le\E[\1_{A_{n,\tau}}\1_{\tilde A_{n,\tau'}}H_n^3|R_k^\tau(\X_n^F)|]+4p_n^3\E[H_n^4]\nonumber\\
&\le\E[\1_{A_{n,\tau}}\1_{\tilde A_{n,\tau'}}H_n^3|R_k^\tau(\X_n^F)|]+CJ_n^{2/3}p_n^3.
\end{align}
In the same manner, we have
\begin{equation}\label{eq:covR_2}
\E[|L_k^\tau(\X_n)R_k^\tau(\X_n^F)|]\E[|L_k^{\tau'}(\X_n)L_k^{\tau'}(\X_n^{F'})|]\\
\le CJ_n^{1/3}p_n\E[\1_{A_{n,\tau}}H_n|R_k^\tau(\X_n^F)|]+CJ_n^{2/3}p_n^3.
\end{equation}
Combining~\eqref{eq:covR_1} and~\eqref{eq:covR_2},
\begin{align}
&\cov(L_k^\tau(\X_n)R_k^\tau(\X_n^F),L_k^{\tau'}(\X_n)L_k^{\tau'}(\X_n^{F'})) \nonumber \\
&\le\E[\1_{A_{n,\tau}}\1_{\tilde A_{n,\tau'}}H_n^3|R_k^\tau(\X_n^F)|] +CJ_n^{1/3}p_n\E[\1_{A_{n,\tau}}H_n|R_k^\tau(\X_n^F)|]+CJ_n^{2/3}p_n^3 \nonumber \\
&\le\E[\1_{A_{n,\tau}}\1_{\tilde A_{n,\tau'}}\E[H_n^3|R_k^\tau(\X_n^F)|\mid\bfb_{n,\tau},\bfb_{n,\tau'}]]+CJ_n^{1/3}p_n\E[\1_{A_{n,\tau}}\E[H_n|R_k^\tau(\X_n^F)|\mid\bfb_{n,\tau},\bfb_{n,\tau'}]] +CJ_n^{2/3}p_n^3 \nonumber \\
\label{e:covlocrembd1} &\le CJ_n^{1/2}(\E[\1_{A_{n,\tau}}\1_{\tilde A_{n,\tau'}}\E[R_k^\tau(\X_n^F)^2\mid\bfb_{n,\tau},\bfb_{n,\tau'}]^{1/2}]+p_n\E[\1_{A_{n,\tau}}\E[R_k^\tau(\X_n^F)^2\mid\bfb_{n,\tau},\bfb_{n,\tau'}]^{1/2}])+CJ_n^{2/3}p_n^3,
\end{align}
where in the third inequality, we use the conditional Cauchy--Schwarz inequality, the mutual independence of $B_n$, $W_n$, $B'_n$, and $W'_n$, and H\"older's inequality.
Since $\tau\notin F$, if $\tau'\notin F$, then
\begin{equation}
\label{e:covlocrembd2}
\1_{A_{n,\tau}}\E[R_k^\tau(\X_n^F)^2\mid\bfb_{n,\tau},\bfb_{n,\tau'}]
=\1_{A_{n,\tau}}\E[R_k^\tau(\X_n^F)^2\mid\bfb_{n,\tau},b_{n,\tau'}]
=\1_{A_{n,\tau}}\E[R_k^\tau(\X_n)^2\mid\bfb_{n,\tau},b_{n,\tau'}]
\le\1_{A_{n,\tau}}\dl_{k,n}
\end{equation}
from the assumption.
If $\tau'\in F$, then it also follows from the assumption that
\begin{align}
&\1_{A_{n,\tau}}\E[R_k^\tau(\X_n^F)^2\mid\bfb_{n,\tau},\bfb_{n,\tau'}] \nonumber \\
&=\1_{A_{n,\tau}}\E[R_k^\tau(\X_n^F)^2\mid\bfb_{n,\tau},b'_{n,\tau'}] \nonumber \\
&=\sum_{\substack{(i_1,i_2)\in\{0,1\}^2\\i_1+i_2=1}}\sum_{i_3\in\{0,1\}}\E[R_k^\tau(\X_n^F)^2\mid\bfb_{n,\tau}=(i_1,i_2),b'_{n,\tau'}=i_3]\1_{\{\bfb_{n,\tau}=(i_1,i_2),b'_{n,\tau'}=i_3\}} \nonumber \\
&=\sum_{\substack{(i_1,i_2)\in\{0,1\}^2\\i_1+i_2=1}}\sum_{i_3\in\{0,1\}}\E[R_k^\tau(\X_n)^2\mid\bfb_{n,\tau}=(i_1,i_2),b_{n,\tau'}=i_3]\1_{\{\bfb_{n,\tau}=(i_1,i_2),b'_{n,\tau'}=i_3\}} \nonumber \\
\label{e:covlocrembd3} &\le\1_{A_{n,\tau}}\dl_{k,n}.
\end{align}
Substituting \eqref{e:covlocrembd2} and \eqref{e:covlocrembd3} into \eqref{e:covlocrembd1} completes the proof.
\end{proof}
Next, we turn to estimate the first term in the right-hand side of~\eqref{eq:16terms}.
Using conditional covariance formula, we can write
\begin{align}\label{eq:covL}
&\cov(L_k^\tau(\X_n)L_k^\tau(\X_n^F),L_k^{\tau'}(\X_n)L_k^{\tau'}(\X_n^{F'}))\nonumber\\
&=\cov(\1_{\tilde A_{n,\tau}}L_k^\tau(\X_n)L_k^\tau(\X_n^F),\1_{\tilde A_{n,\tau'}}L_k^{\tau'}(\X_n)L_k^{\tau'}(\X_n^{F'}))\nonumber\\
&=\cov(\1_{\tilde A_{n,\tau}}\E[L_k^\tau(\X_n)L_k^\tau(\X_n^F)\mid\bfb_{n,\tau},\bfb_{n,\tau'}],\1_{\tilde A_{n,\tau'}}\E[L_k^{\tau'}(\X_n)L_k^{\tau'}(\X_n^{F'})\mid\bfb_{n,\tau},\bfb_{n,\tau'}])\nonumber\\
&\qad+\E[\1_{\tilde A_{n,\tau}}\1_{\tilde A_{n,\tau'}}\cov(L_k^\tau(\X_n)L_k^\tau(\X_n^F),L_k^{\tau'}(\X_n)L_k^{\tau'}(\X_n^{F'})\mid\bfb_{n,\tau},\bfb_{n,\tau'})].
\end{align}
Since $|L_k^\tau(\X_n)L_k^\tau(\X_n^F)|\le\1_{\tilde A_{n,\tau}} H(w_{n,\tau},w'_{n,\tau})^2$ and $|L_k^{\tau'}(\X_n)L_k^{\tau'}(\X_n^{F'})|\le\1_{\tilde A_{n,\tau'}} H(w_{n,\tau'},w'_{n,\tau'})^2$, we immediately have
\[
\cov(L_k^\tau(\X_n)L_k^\tau(\X_n^F),L_k^{\tau'}(\X_n)L_k^{\tau'}(\X_n^{F'})\mid\bfb_{n,\tau},\bfb_{n,\tau'})
\le2J_n^{2/3}\1_{\tilde A_{n,\tau}}\1_{\tilde A_{n,\tau'}}.
\]
Therefore, by a straightforward calculation, the second term in the right-hand side of~\eqref{eq:covL} is estimated as follows:
\begin{align}\label{eq:covL_2nd}
&\E[\1_{\tilde A_{n,\tau}}\1_{\tilde A_{n,\tau'}}\cov(L_k^\tau(\X_n)L_k^\tau(\X_n^F),L_k^{\tau'}(\X_n)L_k^{\tau'}(\X_n^{F'})\mid\bfb_{n,\tau},\bfb_{n,\tau'})]\nonumber\\
&\le\E[\1_{A_{n,\tau}}\1_{A_{n,\tau'}}\cov(L_k^\tau(\X_n)L_k^\tau(\X_n^F),L_k^{\tau'}(\X_n)L_k^{\tau'}(\X_n^{F'})\mid\bfb_{n,\tau},\bfb_{n,\tau'})]+CJ_n^{2/3}p_n^3\nonumber\\
&\le Cp_n^2\rho_{k,n}+CJ_n^{2/3}p_n^3.
\end{align}
Here, we also used $\1_{A_{n,\tau}}\1_{A_{n,\tau'}}\cov(L_k^\tau(\X_n)L_k^\tau(\X_n^F),L_k^{\tau'}(\X_n)L_k^{\tau'}(\X_n^{F'})\mid\bfb_{n,\tau},\bfb_{n,\tau'}) \le\1_{A_{n,\tau}}\1_{A_{n,\tau'}}\rho_{k,n}$, which is the assumption in Theorem~\ref{thm:rdm_deriv}.
The following lemma provides a bound on the first term in the right-hand side of~\eqref{eq:covL}.
\begin{lem}\label{lem:covL_1st}
It holds that
\[
\cov(\1_{\tilde A_{n,\tau}}\E[L_k^\tau(\X_n)L_k^\tau(\X_n^F)\mid\bfb_{n,\tau},\bfb_{n,\tau'}],\1_{\tilde A_{n,\tau'}}\E[L_k^{\tau'}(\X_n)L_k^{\tau'}(\X_n^{F'})\mid\bfb_{n,\tau},\bfb_{n,\tau'}])\le CJ_n^{2/3}p_n^2\gm_{k,n}^{1/2}.
\]
\end{lem}
\begin{proof}
We define
\[
Z_k^{\tau,F}(\X_n)\coloneqq\E[L_k^\tau(\X_n)L_k^\tau(\X_n^F)\mid\bfb_{n,\tau},\bfb_{n,\tau'}]
\quad\text{and}\quad
Z_k^{\tau',F'}(\X_n)\coloneqq\E[L_k^{\tau'}(\X_n)L_k^{\tau'}(\X_n^{F'})\mid\bfb_{n,\tau},\bfb_{n,\tau'}].
\]
We also write
\begin{align*}
Z_k^{\tau,F}(\X_n-\tau')&\coloneqq\E[L_k^\tau(\X_n-\tau')L_k^\tau(\X_n^F-\tau')\mid\bfb_{n,\tau},\bfb_{n,\tau'}] =\E[L_k^\tau(\X_n-\tau')L_k^\tau(\X_n^F-\tau')\mid\bfb_{n,\tau}]
\shortintertext{and}
Z_k^{\tau',F'}(\X_n-\tau)
&\coloneqq\E[L_k^{\tau'}(\X_n-\tau)L_k^{\tau'}(\X_n^{F'}-\tau)\mid\bfb_{n,\tau},\bfb_{n,\tau'}]
=\E[L_k^{\tau'}(\X_n-\tau)L_k^{\tau'}(\X_n^{F'}-\tau)\mid\bfb_{n,\tau'}].
\end{align*}
Then, $Z_k^{\tau,F}(\X_n-\tau')$ and $Z_k^{\tau',F'}(\X_n-\tau)$ are independent.
Furthermore, since
\begin{equation}\label{eq:covL_1st_1}
|L_k^\tau(\X_n)L_k^\tau(\X_n^F)|,|L_k^\tau(\X_n-\tau')L_k^\tau(\X_n^F-\tau')|
\le\1_{\tilde A_{n,\tau}} H(w_{n,\tau},w'_{n,\tau})^2
\end{equation}
from the Lipschitzness of $f$, we have
\[
|Z_k^{\tau,F}(\X_n)|\text{, }|Z_k^{\tau,F}(\X_n-\tau')|
\le\1_{\tilde A_{n,\tau}}\E[H(w_{n,\tau},w'_{n,\tau})^2\mid\bfb_{n,\tau},\bfb_{n,\tau'}]
=\1_{\tilde A_{n,\tau}}\E[H(w_{n,\tau},w'_{n,\tau})^2] \le\1_{\tilde A_{n,\tau}}J_n^{1/3}.
\]
Similarly, $|Z_k^{\tau',F'}(\X_n)|$, $|Z_k^{\tau',F'}(\X_n-\tau)|\le\1_{\tilde A_{n,\tau'}}J_n^{1/3}$.
Noting that
\[
\cov(\1_{\tilde A_{n,\tau}}Z_k^{\tau,F}(\X_n-\tau'),\1_{\tilde A_{n,\tau'}}Z_k^{\tau',F'}(\X_n-\tau))=0,
\]
we can write
\begin{align}\label{eq:covL_1st_2}
&\cov(\1_{\tilde A_{n,\tau}}Z_k^{\tau,F}(\X_n),\1_{\tilde A_{n,\tau'}}Z_k^{\tau',F'}(\X_n))\nonumber\\
&=\cov(\1_{\tilde A_{n,\tau}}(Z_k^{\tau,F}(\X_n)-Z_k^{\tau,F}(\X_n-\tau')),\1_{\tilde A_{n,\tau'}}Z_k^{\tau',F'}(\X_n))\nonumber\\
&\qad+\cov(\1_{\tilde A_{n,\tau}}Z_k^{\tau,F}(\X_n-\tau'),\1_{\tilde A_{n,\tau'}}(Z_k^{\tau',F'}(\X_n)-Z_k^{\tau',F'}(\X_n-\tau))).
\end{align}
We will estimate the first term in the right-hand side of~\eqref{eq:covL_1st_2} as the other term is estimated in the same way.
Using $|Z_k^{\tau',F'}(\X_n)|\le\1_{\tilde A_{n,\tau'}}J_n^{1/3}$, we have
\begin{align*}
&\cov(\1_{\tilde A_{n,\tau}}(Z_k^{\tau,F}(\X_n)-Z_k^{\tau,F}(\X_n-\tau')),\1_{\tilde A_{n,\tau'}}Z_k^{\tau',F'}(\X_n))\\
&\le\E[\1_{\tilde A_{n,\tau}}|Z_k^{\tau,F}(\X_n)-Z_k^{\tau,F}(\X_n-\tau')|\1_{\tilde A_{n,\tau'}}|Z_k^{\tau',F'}(\X_n)|]\\
&\qad+\E[\1_{\tilde A_{n,\tau}}|Z_k^{\tau,F}(\X_n)-Z_k^{\tau,F}(\X_n-\tau')|]\E[\1_{\tilde A_{n,\tau'}}|Z_k^{\tau',F'}(\X_n)|]\\
&\le J_n^{1/3}\E[\1_{\tilde A_{n,\tau}}\1_{\tilde A_{n,\tau'}}|Z_k^{\tau,F}(\X_n)-Z_k^{\tau,F}(\X_n-\tau')|] +2J_n^{1/3}p_n\E[\1_{\tilde A_{n,\tau}}|Z_k^{\tau,F}(\X_n)-Z_k^{\tau,F}(\X_n-\tau')|]\\
&\le J_n^{1/3}\E[\1_{\tilde A_{n,\tau}}\1_{\tilde A_{n,\tau'}}\E[|L_k^\tau(\X_n)L_k^\tau(\X_n^F)-L_k^\tau(\X_n-\tau')L_k^\tau(\X_n^F-\tau')|\mid\bfb_{n,\tau},\bfb_{n,\tau'}]]\\
&\qad+2J_n^{1/3}p_n\E[\1_{\tilde A_{n,\tau}}\E[|L_k^\tau(\X_n)L_k^\tau(\X_n^F)-L_k^\tau(\X_n-\tau')L_k^\tau(\X_n^F-\tau')|\mid\bfb_{n,\tau},\bfb_{n,\tau'}]].
\end{align*}
We can estimate $\E[|L_k^\tau(\X_n)L_k^\tau(\X_n^F)-L_k^\tau(\X_n-\tau')L_k^\tau(\X_n^F-\tau')|\mid\bfb_{n,\tau},\bfb_{n,\tau'}]$ in the both terms in the right-hand side of the above equation as follows.
From~\eqref{eq:covL_1st_1},
\begin{align*}
&\E[|L_k^\tau(\X_n)L_k^\tau(\X_n^F)-L_k^\tau(\X_n-\tau')L_k^\tau(\X_n^F-\tau')|\mid\bfb_{n,\tau},\bfb_{n,\tau'}]\\
&\le2\1_{\tilde A_{n,\tau}}\E[H(w_{n,\tau},w'_{n,\tau})^2\1_{\{L_k^\tau(\X_n)L_k^\tau(\X_n^F)\neq L_k^\tau(\X_n-\tau')L_k^\tau(\X_n^F-\tau')\}}\mid\bfb_{n,\tau},\bfb_{n,\tau'}]\\
&\le2\1_{\tilde A_{n,\tau}}\E[H(w_{n,\tau},w'_{n,\tau})^4]^{1/2}\P(L_k^\tau(\X_n)L_k^\tau(\X_n^F)\neq L_k^\tau(\X_n-\tau')L_k^\tau(\X_n^F-\tau')\mid\bfb_{n,\tau},\bfb_{n,\tau'})^{1/2}\\
&\le2J_n^{1/3}\1_{\tilde A_{n,\tau}}\P(L_k^\tau(\X_n)L_k^\tau(\X_n^F)\neq L_k^\tau(\X_n-\tau')L_k^\tau(\X_n^F-\tau')\mid\bfb_{n,\tau},\bfb_{n,\tau'})^{1/2}.
\end{align*}
In the third line, we use the conditional Cauchy--Schwarz inequality.
As seen in~\eqref{eq:prob_LL-LL}, we have
\[
\P(L_k^\tau(\X_n)L_k^\tau(\X_n^F)\neq L_k^\tau(\X_n-\tau')L_k^\tau(\X_n^F-\tau')\mid\bfb_{n,\tau},\bfb_{n,\tau'})
\le2(d+1)^2\gm_{k,n}\1_{\tilde A_{n,\tau}}.
\]
Therefore, combining the above estimates yields
\[
\cov(\1_{\tilde A_{n,\tau}}(Z_k^{\tau,F}(\X_n)-Z_k^{\tau,F}(\X_n-\tau')),\1_{\tilde A_{n,\tau'}}Z_k^{\tau',F'}(\X_n))\le CJ_n^{2/3}p_n^2\gm_{k,n}^{1/2},
\]
which completes the proof.
\end{proof}

We are now ready to prove Lemma~\ref{lem:disjoint}.
\begin{proof}[Proof of Lemma~\ref{lem:disjoint}]
Substituting bounds from Lemma~\ref{lem:covL_1st} and~\eqref{eq:covL_2nd} in~\eqref{eq:covL}, we obtain
\[
\cov(L_k^\tau(\X_n)L_k^\tau(\X_n^F),L_k^{\tau'}(\X_n)L_k^{\tau'}(\X_n^{F'}))
\le CJ_n^{2/3}p_n^2\gm_{k,n}^{1/2}+Cp_n^2\rho_{k,n}+CJ_n^{2/3}p_n^3.
\]
Consequently, by~\eqref{eq:16terms} and Lemma~\ref{lem:covR},
\begin{align*}
&\cov(\Dl_\tau f(\X_n)\Dl_\tau f(\X_n^F),\Dl_{\tau'} f(\X_n)\Dl_{\tau'} f(\X_n^{F'}))\\
&\le\{CJ_n^{2/3}p_n^2\gm_{k,n}^{1/2}+Cp_n^2\rho_{k,n}+CJ_n^{2/3}p_n^3\}+\{CJ_n^{1/2}p_n^2\dl_{k,n}^{1/2} +CJ_n^{2/3}p_n^3\}\\
&=CJ_n^{2/3}p_n^2\gm_{k,n}^{1/2}+Cp_n^2\rho_{k,n}+CJ_n^{2/3}p_n^3+CJ_n^{1/2}p_n^2\dl_{k,n}^{1/2},
\end{align*}
which completes the proof.
\end{proof}

\section*{Acknowledgements}
SK was supported by a JSPS Grant-in-Aid for Scientific Research (A) (JP20H00119). DY was partially supported by SERB-MATRICS Grant MTR/2020/000470 and CPDA from the Indian Statistical Institute. We are extremely thankful to two anonymous referees for their careful reading and spotting an error in Proposition \ref{prop:gm_est}. 


\end{document}